\documentclass[12pt,a4paper]{article}
\usepackage[a4paper,top=25mm,bottom=25mm,left=25mm,right=25mm]{geometry}
\RequirePackage{iftex} 
\ifXeTeX
\else
  \errmessage{This document must be compiled with XeLaTeX}
\fi
\usepackage{mathtools}
\usepackage{amsmath,amssymb,amsthm,mathrsfs}
\usepackage{microtype}
\usepackage{enumitem}
\usepackage{booktabs}
\usepackage{array,xcolor}
\usepackage{hyperref}
\usepackage[backend=biber,style=alphabetic,sorting=nyt,doi=false,url=false,isbn=false,giveninits=true]{biblatex}
\hypersetup{
  hidelinks,
  bookmarksopen=true,
  bookmarksopenlevel=2,
  bookmarksnumbered=true,
}

\allowdisplaybreaks
\numberwithin{equation}{section}
\setlist[enumerate]{itemsep=3pt,topsep=4pt}
\setlist[itemize]{itemsep=3pt,topsep=4pt}

\usepackage{amsthm}
\numberwithin{equation}{section} 

\newtheorem{theorem}{Theorem}[section]
\newtheorem{lemma}[theorem]{Lemma}
\newtheorem{proposition}[theorem]{Proposition}
\newtheorem{corollary}[theorem]{Corollary}

\theoremstyle{definition}
\newtheorem{definition}[theorem]{Definition}

\newtheorem{remark}[theorem]{Remark}

\newcommand{\G}{\mathbb G}
\newcommand{\Lop}{\mathcal L}
\newcommand{\E}{\mathcal E}
\newcommand{\K}{\mathcal K}
\newcommand{\Rvec}{\mathbf R}
\newcommand{\gradH}{\nabla_{H}}
\newcommand{\1}{\mathbf 1}

\title{\Large\bf A $p = 2$ dichotomy for uniform Riesz transform bounds on stratified Lie groups}
\author{Sheng-Chen Mao \and Yaojun Wang \and Ye Zhang}
\date{}

\begin{document}

\maketitle
\thispagestyle{plain}
\pagestyle{plain}

\begin{abstract}
Let \(\G\) be a stratified Lie group and \(\Lop\) be its
sub-Laplacian. We prove that the full horizontal Riesz transform
$\gradH\Lop^{-1/2}$
is of weak type \((1,1)\) on real-valued functions, with constant at
most \(2\). In particular, the constant is independent of the horizontal dimension,
the homogeneous dimension, the step, and the underlying group structure  of \(\G\). Our result provides a noncommutative generalization of the dimension-free
Euclidean theorem of Ouyang, Spector, and Stockdale \cite{OSS26}, with
the same universal constant. Our proof relies upon a fractional obstacle problem adapted to stratified Lie groups by using the functional calculus of $\Lop$ instead of the Fourier transform.
As a consequence, by interpolation we obtain uniform $L^p$ bounds for the full horizontal Riesz transform $\gradH\Lop^{-1/2}$ for $p \in (1,2]$. 
By contrast, for every $p > 2$, we construct a sequence of stratified Lie groups with fixed horizontal dimension $5$ and steps tending to infinity for which the $L^p$ norms of the horizontal Riesz transforms diverge.
\end{abstract}



\medskip

\section{Introduction}

Standard Calder\'on--Zygmund proofs of endpoint bounds for singular integrals on homogeneous groups inherit constants from the doubling property and therefore from the homogeneous dimension. However, the Riesz transform has an additional structure that is not visible when its components are treated separately: the components combine into an exact $L^2$ isometry. We show that this identity, coupled with an obstacle decomposition that preserves mass, yields a weak type $(1,1)$ constant at most $2$ on every stratified Lie group, independently of the horizontal dimension, homogeneous dimension, step, and group structure.

\medskip

Recall that a connected and simply connected Lie group $\G$ is called a stratified Lie group if its Lie algebra admits the following stratification:
\begin{align}\label{defdeV}
\mathfrak g=V_1\oplus\cdots\oplus V_s,
\qquad [V_1,V_k]=V_{k+1}\quad (1\leq k<s),
\qquad [V_1,V_s]=\{0\}.
\end{align}
The number $s$ is called the step of the stratified Lie group $\G$.
We use $m$ to denote the dimension of $V_1$, which we will also call horizontal dimension.  Fix an inner product on $V_1$, choose an orthonormal basis $X_1,\ldots,X_m$, and identify these vectors with the corresponding left-invariant vector fields. We write
\[
\gradH f=(X_1f,\ldots,X_mf),
\qquad
\Lop=-\sum_{j=1}^{m}X_j^2.
\]
Here $\gradH$ and $\Lop$ are called the horizontal gradient and the sub-Laplacian respectively. Recall that 
$\Lop$ is a nonnegative self-adjoint operator on $L^2(\G)$, where integration is taken with respect to Haar measure. The horizontal vector Riesz transform is
\[
\Rvec f= \Rvec^\G f = \gradH\Lop^{-1/2}f.
\]
Here and in the following when the underlying space is clear, we will drop the superscript $\G$.
The components of $\Rvec$ are classical Calder\'on--Zygmund operators on $\G$; the homogeneous group theory of Folland and Stein \cite{FS82} therefore gives weak type $(1,1)$ estimates with constants that depend on the group. More generally, Alexopoulos \cite{Al92} proved $L^p$ bounds and weak type $(1,1)$ estimates for Riesz transforms on connected Lie groups of polynomial growth. Those qualitative results do not assert uniformity as the underlying group varies.

In Euclidean space, Stein \cite{Ste83, St86} initiated the study of estimates that remain uniform as the dimension tends to infinity and asked whether the weak type constants of the Riesz transforms could be chosen independently of the dimension. Janakiraman \cite{Jan04} obtained an $O(\log n)$ bound. Spector and Stockdale \cite{SS21} subsequently introduced a decomposition based on Dirac masses, and recently Ouyang, Spector, and Stockdale \cite{OSS26}  proved a dimension-free bound with constant at most $2$ for the full vector Riesz transform.

Quantitative estimates of dimension-free type for Riesz transforms on noncommutative models have also received considerable attention. Coulhon, M\"uller, and Zienkiewicz \cite{CMZ96} studied uniform $L^p$ estimates on Heisenberg groups. Lust-Piquard \cite{LP04} obtained similar estimates on groups of Heisenberg-type (H-type for short) and for the CCR heat flow, and then Barbas \cite{Bar10} investigated a better bound in the former case. Bañuelos and Os\k{e}kowski \cite{BO15} used martingale inequalities to derive sharp strong type and related endpoint estimates in geometric settings that include compact Lie groups. These results either address a more specialized class of groups or do not give the universal weak type estimate for the full horizontal vector established here on every stratified group.

\medskip

Our first theorem is the endpoint estimate for the full horizontal Riesz transform.

\begin{theorem}\label{mthm}
Let $\G$ be a stratified Lie group. The operator
$\Rvec=\gradH\Lop^{-1/2}$ extends uniquely to a continuous linear operator from
real-valued $L^1(\G)$ to $L^{1,\infty}(\G;\mathbb R^m)$, and
\begin{equation}\label{wtyp}
 \|\Rvec f\|_{L^{1,\infty}(\G;\mathbb R^m)}
 :=\sup_{\lambda>0}\lambda
 \big|\{x\in\G:|\Rvec f(x)|>\lambda\}\big|
 \leq 2\|f\|_{L^1(\G)}.
\end{equation}
Here $|\cdot|$ (factoring on sets) denotes the Haar measure, and the constant $2$ is independent of $m$, the homogeneous dimension, the step, and the group structure of $\G$.
\end{theorem}

In the abelian case $\G=\mathbb R^m$, Theorem~\ref{mthm} recovers the estimate of Ouyang, Spector, and Stockdale \cite{OSS26}.  While in the noncommutative case,  it yields the same constant on, for instance, Heisenberg and H-type groups, filiform Carnot groups (including the Engel group), and free Carnot groups of arbitrary rank and step. As a result, every component $X_j\Lop^{-1/2}$ therefore satisfies the same estimate.

Our proof of Theorem \ref{mthm} adapts the strategy of Ouyang, Spector, and Stockdale \cite{OSS26} to the noncommutative setting. To be more precise, for $0<\alpha\le2$, we have the decomposition in Theorem \ref{dcmp}. However, the passage is not formal: a general stratified group has neither a simple Fourier multiplier representation nor the rotational symmetry available in Euclidean space. We instead use the functional calculus of the sub-Laplacian $\Lop$ to recover this approach. From now on we will use the notation $D(\mathcal{A})$ to denote the domain of an unbounded operator $\mathcal{A}$. 

\begin{theorem}\label{dcmp}
Let $\lambda>0$ and $0<\alpha\le2$. For every nonnegative $f\in L^1(\G)\cap L^2(\G)$, there are nonnegative functions
$$
\mu\in L^1(\G)\cap L^\infty(\G),
 \qquad
 u\in L^1(\G)\cap D(\Lop^{\alpha/2})
$$
such that
\begin{equation}\label{deco}
 f=\mu+\Lop^{\alpha/2}u,
\end{equation}
and the following assertions hold:
\begin{enumerate}[label=\textup{(\roman*)}]
\item $\min\{f,\lambda\}\leq\mu\leq\lambda$ almost everywhere and
$\|\mu\|_{L^1(\G)}=\|f\|_{L^1(\G)}$;
\item $\mu=\lambda$ almost everywhere on
$\Omega=\{x\in\G:u(x)>0\}$;
\item $\lambda|\Omega|\leq\|f\|_{L^1(\G)}$.
\end{enumerate}
Finally, if $\alpha\geq1$, then $u\in D(\Lop^{1/2})$ and
$\gradH u=0$ almost everywhere on $\G\setminus\Omega$.
\end{theorem}

Theorem~1.2 is in fact a stratified-group generalization of the Euclidean decomposition recently obtained by Ouyang, Spector, and Stockdale \cite[Theorem~1.2]{OSS26}. Although we only need the case $\alpha=1$ for the proof of Theorem~\ref{mthm}, the decomposition is of independent interest so we also prove the result for general $\alpha$. Note that for the sake of completeness, we also include the endpoint case $\alpha = 2$, which is not included in \cite{OSS26}. Decompositions of this kind are related to partial balayage and obstacle problems. In the classical Euclidean setting, it dates back to Kinderlehrer and Stampacchia \cite{KS00}. Then it has been developed in different settings: nonlocal operators on Euclidean spaces \cite{SV13}; compact Riemannian manifolds \cite{GR18}; Heisenberg group \cite{PV13, GRS22}; Dirichlet forms \cite{Kli21}. However, we cannot apply their results directly since their results work mainly on bounded domains (or similar conditions).  We therefore adapt the variational argument to setting of the whole group. In particular, the conservation of mass is recovered by means of a suitable cutoff argument at infinity.

\medskip

Furthermore, by Proposition \ref{l2iso} and the Marcinkiewicz interpolation theorem, we have the following uniform $L^p$ bounds for $p \in (1,2]$.

\begin{corollary}\label{lpbd}
For $1<p\leq2$ there exists a constant $C_p > 0$ which depends only  on $p$ such that for every real-valued $f\in L^p(\G)$ we have
\begin{equation}\label{lpup}
 \|\Rvec f\|_{L^p(\G;\mathbb R^m)}
 \leq C_p \|f\|_{L^p(\G)}.
\end{equation}
In particular, the constant $C_p$ is independent of $m$, the homogeneous dimension, the step, and the group structure of $\G$.
\end{corollary}

Corollary~\ref{lpbd} does not assert the estimates for $p>2$.  Regarding the results in Euclidean spaces \cite{Ste83}, Heisenberg groups \cite{CMZ96}, and H-type groups \cite{LP04, Bar10}, it is therefore natural to ask whether similar results can hold for $p > 2$. However, we will give  a negative answer for the uniform $L^p$ bounds with $p > 2$ for the class of general
stratified Lie groups, which is a new feature of the Riesz transforms on stratified Lie groups (compared to those on Euclidean spaces).

\begin{theorem}\label{thmcce}
For any $p > 2$, there is a sequence of stratified Lie groups $\{\G_k\}_{k = 1}^\infty$ such that
\begin{equation}\label{eq:main-divergence}
  \| \Rvec^{\G_k} \|_{
       L^p(\G_k)\to L^p(\G_k;\mathbb{R}^{m_k})} \to \infty, \qquad \mbox{as} \quad k \to \infty,
\end{equation}
where $m_k$ is the corresponding horizontal dimension of $\G_k$. Consequently, for $p > 2$ there is no finite constant $C_p > 0$ such that
\begin{equation}\label{eq:universal-bound}
  \|\Rvec f\|_{L^p(\G;\mathbb{R}^m)}
  \le C_p\|f\|_{L^p(\G)}
\end{equation}
holds for all stratified Lie groups $\G$ and all $f\in L^p(\G)$.
\end{theorem}

\begin{remark}
We will see in the construction of $\G_k$ that the horizontal dimension of $\G_k$, which is denoted by $m_k$, can be chosen to be constant $5$, and the step of $\G_k$, which is denoted by $s_k$, tends to infinity as $k \to \infty$ (see Remark \ref{restep}). As a consequence, it is still plausible that such uniform $L^p$ bound ($p>2$) still exists for some special class of stratified Lie groups with certain restriction on the structure or the step, besides known results on Euclidean spaces \cite{Ste83} and Heisenberg groups \cite{CMZ96}. In fact, we will provide another affirmative answer in the setting of H-type groups in a forthcoming work \cite{MWZ26}, where the group structure is more involved.
\end{remark}

Nevertheless, it is not surprising that the behavior of Riesz transform changes from the range $1 < p \le 2$ to $p > 2$. See for example \cite{CD99, DOY06} for the results for a single Riesz transform and \cite{LF04, D26} for the uniform bound results of Riesz transforms in different settings. In fact, we will make use of the counterexample in \cite{D26} to give the construction of the sequence $\{\G_k\}_{k = 1}^\infty$.

\medskip

We remark that counterparts of Theorem \ref{mthm} and Corollary \ref{lpbd} are also valid for the complex-valued case. We will not expand it in the Introduction but refer the interested reader to Theorem \ref{mthmC} and Corollary \ref{lpbdC} for more details.

\medskip
The paper is organized as follows.
In Section~\ref{pre} we collect the required facts about stratified groups. In Section~\ref{pf1} we prove Theorem~\ref{mthm} from Theorem~\ref{dcmp}. In Section~\ref{pf2}, we construct the obstacle minimizer, prove the Lewy--Stampacchia estimate, establish conservation of mass and complementarity, and complete the proof of Theorem~\ref{dcmp} with all the materials developed in Section~\ref{pre}. Finally, Section~\ref{sce} is dedicated to give a construction of the sequence $\{\G_k\}_{k = 1}^\infty$ in Theorem~\ref{thmcce}.

\section{Preliminaries}\label{pre}

\subsection{The sub-Laplacian and  Riesz transform on stratified groups}

For $r>0$, define the linear map $\delta_r:\mathfrak g\to\mathfrak g$ by
\[
 \delta_r(v_1+\cdots+v_s):=rv_1+r^2v_2+\cdots+r^sv_s,
 \qquad v_k\in V_k, 1 \le k \le s.
\]
The stratification \eqref{defdeV} implies that $\delta_r$ is a Lie algebra automorphism. Since
$\G$ is connected, simply connected, and nilpotent, via the exponential map it induces a global diffeomorphism on $\G$, which we denote by $\delta_r$ as well.  Let
\[
 Q:=\sum_{k=1}^{s}k\dim V_k
\]
be the homogeneous dimension. A stratified group is unimodular, so its Haar
measure $|\cdot|$ is both left and right invariant; under the preceding dilation it
satisfies $|\delta_rE|=r^Q|E|$. In the following we will use the notation $L^p(\G)$ ($p \in [1,\infty)$) to denote the real $p$-integrable functions with respect to Haar measure while the notation $L^p(\G;\mathbb{R}^n)$ is reserved to vector-valued $L^p$-spaces. In particular, we have $L^p(\G;\mathbb{C})$, the complex-valued $L^p$-spaces. When the situation is clear from the context we will write the $L^p$-norm as $\|\cdot\|_p$. The same conventions apply to the $L^\infty$-spaces, the weak $L^p$-spaces, Sobolev spaces, smooth functions with compact supports, etc. In the rest of this section (except Subsection \ref{ssAhti}) we will focus on the complex-valued case while the results are also true for the real-valued case. We use the convention that the  inner product on $L^2(\G;\mathbb{C})$ is linear in its first
variable:
\[
 \langle v,w\rangle_2=\int_\G v(x)\overline{w(x)}\,dx.
\]
As a consequence of the unimodularity every $X_j$ is skew-adjoint on $C_c^\infty(\G;\mathbb{C})$, and
\begin{equation}\label{grad}
 \langle\Lop v,v\rangle_2
 =\sum_{j=1}^{m}\|X_jv\|_2^2
 =\|\gradH v\|_2^2,
 \qquad v\in C_c^\infty(\G;\mathbb{C}).
\end{equation}
The closure of this form has domain $D(\Lop^{1/2})$, which is the horizontal Sobolev space $W_H^{1,2}(\G;\mathbb{C})$. Standard structural and analytic facts used here can be found in e.g. \cite{Fol75,FS82,VSC92}. The following proposition shows that the Riesz transform is an isometry on $L^2(\G;\mathbb{C})$.

\begin{proposition}\label{l2iso}
The  kernel of $\Lop$ in $L^2(\G;\mathbb{C})$ is $\{0\}$. Moreover, the operators
$\Rvec_j=X_j\Lop^{-1/2}$, initially defined on the range of $\Lop^{1/2}$, extend uniquely to $L^2(\G;\mathbb{C})$ and satisfy
\begin{equation}\label{riso}
 \sum_{j=1}^{m}\|\Rvec_jf\|_2^2=\|f\|_2^2,
 \qquad f\in L^2(\G;\mathbb{C}).
\end{equation}
\end{proposition}

\begin{proof}
If $h\in\ker \Lop$, then $h\in D(\Lop^{1/2})$ and \eqref{grad}, extended by closure, gives $X_jh=0$ in the sense of distributions for every $j$. By hypoellipticity the function $h$ is also smooth.  Repeated commutators also annihilate $h$, while the vector fields $X_1,\ldots,X_m$ and their commutators span $\mathfrak g$. Thus every left-invariant derivative of $h$ vanishes, and thus $h$ is constant on the whole group $\G$. Since a nontrivial stratified group has infinite Haar measure, the function $h$ has to be $0$.

Now define an operator on $\operatorname{Ran}(\Lop^{1/2})$ by
\[
 U(\Lop^{1/2}h)=\gradH h,
 \qquad h\in D(\Lop^{1/2}).
\]
The definition is unambiguous because $\ker\Lop^{1/2}=\ker\Lop=\{0\}$.
The identity \eqref{grad} gives 
\[
 \|U(\Lop^{1/2}h)\|_{2}
 =\|\Lop^{1/2}h\|_2.
\]
Moreover,
$\overline{\operatorname{Ran}(\Lop^{1/2})}=(\ker\Lop^{1/2})^{\perp} = (\ker\Lop)^{\perp}=L^2(\G;\mathbb{C})$.
Thus $U$ extends uniquely to an isometry from $L^2(\G;\mathbb{C})$ into
$L^2(\G;\mathbb C^m)$. We define $\Rvec=U$ and let $\Rvec_j$ be its
$j$th component. This proves \eqref{riso}.

\end{proof}

\subsection{Fractional Dirichlet forms}

Let $P_t=e^{-t\Lop}$ be the semigroup of the generator $\Lop$. With the standard convolution convention, we have
\begin{equation} \label{ff7} 
     P_tv(x)=(v*p_t)(x)
     :=\int_\G v(y)p_t(y^{-1}x)\,dy
     =\int_\G v(xz)p_t(z)\,dz.
\end{equation}
For the last equality, we used a change of variable $y=xz$ and the left invariance of the Haar
measure, and then a change of variables $z \mapsto z^{-1}$ and the unimodularity. We also recall that Folland \cite[Theorem~(3.1)]{Fol75} (see also \cite[\S~1G]{FS82}) proved that $P_t$ is a symmetric, positivity
preserving contraction semigroup whose convolution kernel has mass one and
satisfies
\begin{equation}\label{hsca}
 p_t(x)=t^{-Q/2}p_1(\delta_{t^{-1/2}}x),
 \qquad p_t(x^{-1})=p_t(x), \qquad \forall \, t > 0, x \in \G.
\end{equation}
In addition, $p_1$ is a Schwartz function; see 
\cite[Corollaries (3.5) and (3.6)]{Fol75}.
For $0<\alpha\le2$, define
\[
 \mathsf V_\alpha:=D(\Lop^{\alpha/4}),
 \qquad
 \E_\alpha(v,w)
 :=\langle\Lop^{\alpha/4}v,\Lop^{\alpha/4}w\rangle_2,
 \qquad
 \E_\alpha(v):=\E_\alpha(v,v).
\]
If $E_\lambda$ denotes the spectral resolution of $\Lop$, then
\[
 \mathsf V_\alpha
 =\left\{v\in L^2(\G;\mathbb{C}):
   \int_{[0,\infty)}\lambda^{\alpha/2}
       \,d\langle E_\lambda v,v\rangle_2<\infty\right\},
 \qquad
 \E_\alpha(v)=\|\Lop^{\alpha/4}v\|_2^2.
\]
Thus $\mathsf V_\alpha$ is the finite energy space and is a Hilbert space
for the norm $\bigl(\|v\|_2^2+\E_\alpha(v)\bigr)^{1/2}$, which coincides with the graph norm of $\Lop^{\alpha/4}$.

The following difference kernel representation of $\E_\alpha(\cdot,\cdot) $ will be useful.

\begin{lemma}\label{jrep}
For $\alpha \in (0,2]$, we adopt the convention that when $v \in L^2(\G;\mathbb{C})$ but $v \notin \mathsf V_\alpha$, we have $\E_\alpha(v) =  \infty$. Then the following statements hold.
\begin{enumerate}[label=\textup{(\roman*)}]
\item 
Let $0<\alpha<2$ and put
\[
 c_\alpha=\frac{\alpha/2}{\Gamma(1-\alpha/2)},
 \qquad
 J_\alpha(z)=c_\alpha\int_0^\infty
 p_t(z)\,\frac{dt}{t^{1+\alpha/2}},
 \qquad z\in\G\setminus\{e\}.
\]
Here $e$ denotes the identity element of $\G$.
Then $J_\alpha$ is finite away from the identity, positive, symmetric, and
homogeneous of degree $-Q-\alpha$.
Then for $v \in L^2(\G;\mathbb{C})$ we have 
\begin{equation}\label{jumpp}
 \E_\alpha(v) = \frac12\int_\G\int_\G|v(x)-v(xz)|^2J_\alpha(z)\,dz\,dx.
\end{equation}
Furthermore, for $v,w\in\mathsf V_\alpha$,
\begin{equation}\label{jump}
 \E_\alpha(v,w)
 =\frac12\int_\G\int_\G
 \big(v(x)-v(xz)\big)\overline{\big(w(x)-w(xz)\big)}
 J_\alpha(z)\,dz\,dx,
\end{equation}
and the integral is absolutely convergent.
\item Let $\alpha = 2$, for $v \in L^2(\G;\mathbb{C})$ we have 
\begin{equation}\label{jump22p}
 \E_2(v)
 = \lim_{t \to 0^+}\frac{1}{2t}\int_\G\int_\G
 \big|v(x)-v(xz)\big|^2
 p_t(z)\,dz\,dx,
\end{equation}
where the limit on the right-hand side always exists (could be $\infty$).
Furthermore, for $v,w\in\mathsf V_2$, the following limit holds:
\begin{equation}\label{jump22}
 \E_2(v,w)
 = \lim_{t \to 0^+}\frac{1}{2t}\int_\G\int_\G
 \big(v(x)-v(xz)\big)\overline{\big(w(x)-w(xz)\big)}
 p_t(z)\,dz\,dx.
\end{equation}
\end{enumerate}
\end{lemma}

\begin{proof}
We begin with the first assertion. Thus we have $0 < \alpha < 2$. It is easy to verify the scalar identity:
\begin{equation}\label{scal}
 \lambda^{\alpha/2}
 =c_\alpha\int_0^\infty(1-e^{-t\lambda})
       \frac{dt}{t^{1+{\alpha/2}}},
 \qquad c_\alpha=\frac{{\alpha/2}}{\Gamma(1-{\alpha/2})}.
\end{equation}
Applying \eqref{scal} to the spectral resolution of $\Lop$ gives 
\begin{equation} \label{ff8} 
     \E_\alpha(v)
     =c_\alpha\int_0^\infty
         \langle(I-P_t)v,v\rangle_2\frac{dt}{t^{1+{\alpha/2}}}, \qquad \forall \, v\in L^2(\G;\mathbb{C}).
\end{equation}
Note that the
double integral below is absolutely convergent for every $v \in L^2(\G;\mathbb{C})$
 and
 \begin{equation} \label{ff9}       \int_\G\int_\G|v(x)-v(xz)|^2p_t(z)\,dz\,dx
      =2\langle(I-P_t)v,v\rangle_2\leq4\|v\|_2^2,
 \end{equation}
where we have used that $|v(x)-v(xz)|^2= |v(x)|^2+|v(xz)|^2-2\mathrm{Re}\,({v(x)v(xz)})$,  \eqref{ff7} and the self-adjointness of $P_t$, by direct computation.
Recall the two-sided Gaussian estimates with respect to the
Carnot--Carath\'eodory distance $\rho$ (cf. e.g. \cite[Theorem~VIII.2.9]{VSC92}),
\begin{equation}\label{gaus}
 C_1t^{-Q/2}\exp\!\left(-c_1\frac{\rho(z)^2}{t}\right)
 \leq p_t(z)\leq
 C_2t^{-Q/2}\exp\!\left(-c_2\frac{\rho(z)^2}{t}\right).
\end{equation}
In particular, $p_t(z)>0$. For $z\ne e$,
 \eqref{gaus} yields that $J_\alpha$ is finite and strictly positive.
The value at $e$ is immaterial because a singleton has Haar measure zero;
we may set $J_\alpha(e)=0$. Symmetry follows from
$p_t(z^{-1})=p_t(z)$. Formula \eqref{hsca} and the substitution
$t=r^2\tau$ give
\[
 J_\alpha(\delta_rz)=r^{-Q-\alpha}J_\alpha(z),
\]
which proves the homogeneity of $J_\alpha$. Then inserting \eqref{ff9} into \eqref{ff8}, since all integrands are nonnegative, Fubini's theorem permits the
$t$ and spatial integrations to be interchanged and this yields \eqref{jumpp}. 
A standard complex polarization verifies \eqref{jump}. Finally, Cauchy--Schwarz with
respect to $J_\alpha(z)\,dz\,dx$ proves absolute convergence for every
$v,w\in\mathsf V_\alpha$. 

Now we are left to prove the second assertion, namely $\alpha = 2$. Notice that, for every $\lambda \ge 0$, as $t \to 0^+$, we have $ \frac{1 - e^{-\lambda t}}{t} (\ge 0)$ increases to $\lambda$. By the monotone convergence theorem, we obtain
\[
\E_2(v) = \lim_{t \to 0^+} \frac{1}{t} \langle(I-P_t)v,v\rangle_2. 
\]
Then inserting \eqref{ff9} into the equation above we obtain \eqref{jump22p} and then a standard complex polarization gives \eqref{jump22}.
\end{proof}

The next lemma is an immediate consequence of Lemma  \ref{jrep}, which shows that $\E_\alpha$ is indeed a Dirichlet form (see for example \cite{FOT11}).

\begin{lemma}\label{mark}
Let $0<\alpha\le 2$. The following properties hold:
\begin{enumerate}[label=\textup{(\roman*)}]
\item If $\Phi:\mathbb C\to\mathbb C$ is $1$-Lipschitz and $\Phi(0)=0$, then
$\Phi\circ v\in\mathsf V_\alpha$ and
$\E_\alpha(\Phi\circ v)\leq\E_\alpha(v)$ for every
$v\in\mathsf V_\alpha$.
\item If $v,w\in\mathsf V_\alpha$ are nonnegative and $vw=0$ almost
everywhere, then $\E_\alpha(v,w)\leq0$.
\item For every real-valued $v\in\mathsf V_\alpha$,
\[
 \E_\alpha(v,v^-)
 \leq-\E_\alpha(v^-)
 \leq0,
\]
where $v^- := -\min\{v,0\}$.
\end{enumerate}
\end{lemma}

\begin{proof}
For (i), note that $|\Phi(v)|\leq|v|$ because $\Phi(0)=0$, so
$\Phi(v)\in L^2(\G;\mathbb{C})$. Formula \eqref{jumpp} for $\alpha \in (0,2)$ or formula \eqref{jump22p} for $\alpha = 2$, and
$|\Phi(a)-\Phi(b)|\leq|a-b|$ show that
$\Phi(v)\in\mathsf V_\alpha$. If $v,w\geq0$ and $vw=0$, then
\[
(v(x)-v(xz))\overline{(w(x)-w(xz)) } 
 =-v(x)w(xz)-v(xz)w(x)\leq0,
\]
which proves (ii). Finally, for real-valued $v$, the contractions
$v\mapsto v^+ := 
\max\{v,0\}$, $v\mapsto v^- := -\min\{v,0\}$, and (i) show that $v^\pm\in\mathsf V_\alpha$. Since
$v=v^+-v^-$, applying (ii) to $v^+$ and $v^-$ gives (iii).
\end{proof}

\subsection{Cores, cutoffs, and interpolation}

The operator $\Lop$ on $C_c^\infty(\G;\mathbb{C})$ is essentially self-adjoint. We
need the following precise fractional core property:
\begin{equation}\label{core}
 \overline{C_c^\infty(\G;\mathbb{C})}^{\,\sqrt{\|\cdot\|_2^2+\E_\alpha(\cdot)}}
 =\mathsf V_\alpha.
\end{equation}
Indeed, in Folland's notation the Sobolev space $S_a^2$ is
$D(\Lop^{a/2})$ with its graph norm. Thus,
\cite[Theorem~(4.5)]{Fol75}, applied with $p=2$ and $a=\alpha/2$, 
states exactly that $C_c^\infty(\G;\mathbb{C})$ is dense in
$D(\Lop^{\alpha/4})=\mathsf V_\alpha$ for the form norm or equivalently the graph norm.  

\medskip

The following cutoff estimate is needed to recover mass on the whole group in the proof of Theorem~\ref{dcmp}.

\begin{lemma}\label{cut}
Let $0<\alpha\le2$. Choose $\chi\in C_c^\infty(\G)$ such that
$0\leq\chi\leq1$ and $\chi=1$ in a neighborhood of the identity, and set
$\chi_R=\chi\circ\delta_{R^{-1}}$ for $R\geq1$. Then
$\chi_R(x)\to1$ for every $x\in\G$ and
\begin{equation}\label{cest}
 \|\Lop^{\alpha/2}\chi_R\|_\infty
 =R^{-\alpha}\|\Lop^{\alpha/2}\chi\|_\infty
 \leq C_{\chi,\alpha}R^{-\alpha}.
\end{equation}
\end{lemma}

\begin{proof}
Note that
$\chi\in D(\Lop^{\alpha/2})$. We first show that $\Lop^{\alpha/2} \chi \in L^\infty (\G)$. For $\alpha = 2$, it is clear. In the remaining case $\alpha \in (0,2)$, applying \eqref{scal} to the spectral resolution of $\Lop$ gives the following formula for fractional powers:
\begin{equation}\label{bal}
 \Lop^{\alpha/2}\chi
 =c_\alpha\int_0^\infty (\chi-P_t \chi)\,
 \frac{dt}{t^{1+\alpha/2}}.
\end{equation}
The integral converges in $L^2(\G;\mathbb{C})$ and it also converges in $L^\infty(\G;\mathbb{C})$, since the contractivity of
$P_t$ on $L^\infty(\G;\mathbb{C})$ gives
\begin{align}\label{estdif}
 \|\chi-P_t\chi\|_\infty
 \leq
 \begin{cases}
 t\|\Lop\chi\|_\infty,&0<t\leq1,\\
 2\|\chi\|_\infty,&t>1.
 \end{cases}
\end{align}
To be more precise, to prove \eqref{estdif}, for the case  $0 < t\leq1$ one uses
$\chi-P_t\chi=\int_0^tP_r\Lop\chi\,dr$.
The integral in  \eqref{bal} near zero is finite because $\alpha<2$, and  near
infinity is finite because $\alpha>0$. Now if
$D_Rh=h\circ\delta_{R^{-1}}$, then
$U_R := R^{-Q/2}D_R$ is unitary on $L^2(\G;\mathbb{C})$ and
$\Lop D_R=R^{-2}D_R\Lop$ on $C_c^\infty(\G;\mathbb{C})$, which implies
\[
U_R^{-1} \Lop U_R = R^{-2}  \Lop, \qquad \forall \, R > 0.
\]Then the spectral calculus therefore gives
\[
U_R^{-1} \Lop^{\alpha/2} U_R = R^{-\alpha} \Lop^{\alpha/2}, \qquad \forall \, \alpha \in (0,2], R > 0,
\]
which is exactly
\[
 \Lop^{\alpha/2}(\chi\circ\delta_{R^{-1}})
 =R^{-\alpha}(\Lop^{\alpha/2}\chi)\circ\delta_{R^{-1}},  \qquad \forall \, \alpha \in (0,2], R > 0.
\]
This proves \eqref{cest}.
\end{proof}

We shall also use
the spectral interpolation inequality
\begin{equation}\label{intr}
 \|\Lop^\theta h\|_2
 \leq \|h\|_2^{1-\theta}\|\Lop h\|_2^\theta,
 \qquad h\in D(\Lop),\quad 0\le \theta \le1,
\end{equation}
which follows immediately from the spectral theorem and H\"older's
inequality for the spectral measure. Furthermore we need the following Nash estimate in the proof of Theorem \ref{dcmp}.

\begin{lemma}\label{nash}
For $0<\alpha\le 2$, there is a finite constant $C_{\G,\alpha}$ such that
\begin{equation}\label{nshq}
 \|v\|_2^2
 \leq C_{\G,\alpha}
 \|v\|_1^{\frac{2\alpha}{Q+\alpha}}
 \E_\alpha(v)^{\frac{Q}{Q+\alpha}}
\end{equation}
for every $v\in L^1(\G;\mathbb{C})\cap\mathsf V_\alpha$.
\end{lemma}

\begin{proof}
From Young's convolution inequality and the first equation of \eqref{hsca} it follows that
\[
 \|P_t v\|_\infty
 \leq\|p_t\|_\infty\|v\|_1
 =t^{-Q/2}\|p_1\|_\infty\|v\|_1.
\]
Hence $\langle P_tv,v\rangle_2\leq  \|P_t v\|_\infty \|v\|_1 \le 
Ct^{-Q/2}\|v\|_1^2$. The spectral theorem and the elementary bound
$1-e^{-r}\leq C_\alpha r^{\alpha/2}$ give
\[
 \begin{split}
 \|v\|_2^2
 &=\langle P_tv,v\rangle_2
   +\langle(I-P_t)v,v\rangle_2\\
 &\leq Ct^{-Q/2}\|v\|_1^2
   +C_\alpha t^{\alpha/2}\E_\alpha(v).
 \end{split}
\]
We split the proof into three cases: (i). if $\|v\|_1\E_\alpha(v) > 0$,
optimizing over $t>0$ proves \eqref{nshq}; (ii).  If $\|v\|_1=0$, then $v=0$ and \eqref{nshq} holds automatically; (iii). If  $\E_\alpha(v)=0$, we have $\Lop^{\alpha/4} v = 0$, which implies $\Lop v = 0$. Proposition~\ref{l2iso} (see also \eqref{jumpp} and \eqref{jump22p}) gives $v=0$ again. 
\end{proof}

\subsection{A horizontal truncation identity}\label{ssAhti}

Recall that we use the notation $v^+ := 
\max\{v,0\}$ and $v^- := -\min\{v,0\}$. Moreover, we use $\1_A$ to denote the characteristic function of a set $A$, namely
\[
\1_A(x) := \begin{cases}
    1, \qquad x \in A, \\
    0, \qquad x \notin A.
\end{cases}
\]
We shall also use the following horizontal truncation lemma. 

\begin{lemma}\label{stam}
If $v\in W_{H,\mathrm{loc}}^{1,1}(\G)$ is real-valued and $c\in\mathbb R$,
then
$(v-c)^\pm\in W_{H,\mathrm{loc}}^{1,1}(\G)$ and
\[
 \gradH(v-c)^+=\1_{\{v>c\}}\gradH v,
 \qquad
 \gradH(v-c)^-=-\1_{\{v<c\}}\gradH v.
\]
Consequently, $\gradH v=0$ almost everywhere on $\{v=c\}$.
\end{lemma}

\begin{proof}
Let $\Omega_0\Subset\G$ be an arbitrary open set  and put $q=v-c$. Then
$q\in W_H^{1,1}(\Omega_0)$. In exponential coordinates, the horizontal
fields are smooth and satisfy H\"ormander's condition, so \cite[Proposition~5.4]{Gar16} with $p=1$ therein, implies that
\[
 \gradH q^+=\1_{\{q>0\}}\gradH q,
 \qquad
 \gradH q^-=-\1_{\{q<0\}}\gradH q
\]
almost everywhere in $\Omega_0$. Since $q=q^+-q^-$, subtracting these two
identities shows that $\gradH q=0$ almost everywhere on $\{q=0\}$.
Equivalently, $\gradH v=0$ almost everywhere on $\{v=c\}\cap\Omega_0$.
Because $\Omega_0\Subset\G$ was arbitrary, all three assertions hold globally
on $\G$.
\end{proof}

\section{Proof of Theorem  \ref{mthm}}\label{pf1}

We now derive the endpoint estimate from Theorem~\ref{dcmp}. The argument shows why exact mass conservation and localization of the horizontal gradient are the two essential features of the decomposition.

\begin{proof}[Proof of Theorem~\ref{mthm}]
First suppose that $f\in L^1(\G)\cap L^2(\G)$ is real-valued. Write
$f=f^+-f^-$ with $f^+ := 
\max\{f,0\}$, $f^- := -\min\{f,0\}$, and fix $\lambda>0$. Note that $f^\pm \in L^1(\G)\cap L^2(\G)$. Applying Theorem~\ref{dcmp} with $\alpha=1$ and level $\lambda$ separately to $f^+$ and $f^-$ respectively, we obtain
\[
 f^\pm=\mu_\pm+\Lop^{1/2}u_\pm,
 \qquad
 0\leq\mu_\pm\leq\lambda,
\]
and measurable sets $\Omega_\pm=\{u_\pm>0\}$ such that
\[
 \lambda|\Omega_\pm|\leq\|f^\pm\|_1,
 \qquad
 \gradH u_\pm=0
 \quad\text{almost everywhere on }\G\setminus\Omega_\pm.
\]
Set
\[
 \mu:=\mu_+-\mu_-,
 \qquad u:=u_+-u_-,
 \qquad \Omega:=\Omega_+\cup\Omega_-.
\]
By the construction of the Riesz transform in Proposition~\ref{l2iso}, we get
\begin{equation}
 \Rvec f=\Rvec\mu+\gradH u.
\end{equation}
Since $\gradH u=0$ almost everywhere on $\G\setminus\Omega$, we have up to a set of measure zero
\[
 \{x:|\Rvec f(x)|>\lambda\}
 \subset\Omega\cup\{x:|\Rvec\mu(x)|>\lambda\}.
\]
Moreover,
\begin{equation}\label{omeg}
 \lambda|\Omega|
 \leq\|f^+\|_1+\|f^-\|_1
 =\|f\|_1.
\end{equation}
Because $0\leq\mu_\pm\leq\lambda$, we have $|\mu|\leq\lambda$ and
\[
 \int_\G|\mu|
 \leq\int_\G\mu_++\int_\G\mu_-
 =\|\mu_+\|_1+\|\mu_-\|_1
 =\|f^+\|_1+\|f^-\|_1
 =\|f\|_1.
\]
Here the penultimate equality follows from
Theorem~\ref{dcmp} (i). Consequently,
Proposition~\ref{l2iso} and Chebyshev's inequality imply
\begin{equation}\label{good}
 \lambda\big|\{x:|\Rvec\mu(x)|>\lambda\}\big|
 \leq\lambda^{-1}\|\Rvec\mu\|_2^2
 =\lambda^{-1}\|\mu\|_2^2 \leq \|\mu\|_1 
 \leq\|f\|_1.
\end{equation}
Combining \eqref{omeg} and \eqref{good} proves \eqref{wtyp} for
$f\in L^1\cap L^2$. A standard extension argument gives the rest of the assertion.

\end{proof}

Furthermore, we have the counterparts of Theorem \ref{mthm} and Corollary \ref{lpbd}.

\begin{theorem}\label{mthmC}
Let $\G$ be a stratified Lie group. The operator
$\Rvec=\gradH\Lop^{-1/2}$ extends uniquely to a continuous linear operator from
$L^1(\G;\mathbb{C})$ to $L^{1,\infty}(\G; \mathbb{C}^m)$, and
\begin{equation}\label{wtypC}
 \|\Rvec f\|_{L^{1,\infty}(\G;\mathbb C^m)}
 :=\sup_{\lambda>0}\lambda
 \big|\{x\in\G:|\Rvec f(x)|>\lambda\}\big|
 \leq 4\|f\|_{L^1(\G;\mathbb{C})}.
\end{equation}
The constant $4$ is independent of $m$, the homogeneous dimension, the step, and the group structure of $\G$.
\end{theorem}

\begin{proof}
For any $f \in L^1(\G;\mathbb{C})$, we can write $f = a + ib$ with $a,b \in L^1(\G)$. Thus we obtain
\[
 \{|\Rvec f|>\lambda\}
 \subset
 \{|\Rvec a|>\lambda/\sqrt2\}
 \cup
 \{|\Rvec b|>\lambda/\sqrt2\}.
\]
Applying the estimate for real-valued functions and using
$\|a\|_1+\|b\|_1\leq\sqrt2\|f\|_1$ show the validity of \eqref{wtypC}.
\end{proof}

\begin{corollary}\label{lpbdC}
For $1<p<2$ there exists a constant $D_p > 0$ which depends only  on $p$ such that for  $f\in L^p(\G;\mathbb{C})$ we have
\begin{equation}\label{lpupC}
 \|\Rvec f\|_{L^p(\G;\mathbb C^m)}
 \leq D_p \|f\|_{L^p(\G;\mathbb{C})}.
\end{equation}
In particular, the constant $D_p$ is independent of $m$, the homogeneous dimension, the step, and the group structure of $\G$.
\end{corollary}

\section{Proof of Theorem  \ref{dcmp}}\label{pf2}

Throughout this section, we assume that $\lambda>0$, $0<\alpha\le2$ and $f$ is a fixed nonnegative function in $L^1(\G)\cap L^2(\G)$. All the function spaces in this section are real-valued (except those in Proposition \ref{reg}).

\medskip

Recall that $\mathsf{V}_\alpha= D(\Lop^{\alpha/4})$. We introduce the space
\begin{align*}
\K_\alpha:=\{v\in L^1(\G)\cap\mathsf V_\alpha:v\geq0\text{ almost everywhere}\},  
\end{align*}
and the functional 
\begin{equation}\label{engy}
	\mathcal J_\lambda^\alpha(v)
	:=\frac12\E_\alpha(v)-\int_\G(f-\lambda)v
	=\frac12\E_\alpha(v)-\int_\G fv+\lambda\|v\|_1, \qquad v\in\K_\alpha.
\end{equation}

\subsection{The minimizer and its variational inequality}

In the following we will use the direct method in the calculus of variations. To this end, we first establish the coercivity of the functional $\mathcal J_\lambda^\alpha$.

\begin{proposition}\label{coer}
For every $v\in \K_\alpha$, one has
\begin{equation}\label{coeq}
 \mathcal J_\lambda^\alpha(v)
 \geq\frac14\E_\alpha(v)
 +\frac\lambda2\|v\|_1-C_0,
\end{equation}
where
$$
C_0=\frac{Q^2}{(Q+\alpha)^2}\left(\frac{2\alpha}{\lambda(Q+\alpha)}\right)^{2\alpha/Q}(C_{\G,\alpha}\|f\|_2^2)^{(Q+\alpha)/Q},
$$
with the constant $C_{\G,\alpha}$ given in Lemma \ref{nash}.
\end{proposition}

\begin{proof}
Using the Cauchy--Schwarz inequality and Lemma \ref{nash}, we obtain
\begin{equation*}
\int_\G fv\leq C_{\G,\alpha}^{1/2}\|f\|_2
\|v\|_1^{\frac{\alpha}{Q+\alpha}}
\E_\alpha(v)^{\frac{Q}{2(Q+\alpha)}},\qquad \forall \, v\in \K_\alpha.
\end{equation*}	 
Set $a=\frac{\alpha}{Q+\alpha}$ and $b=\frac{Q}{2(Q+\alpha)}$. Then $\gamma:=1-a-b=\frac{Q}{2(Q+\alpha)}>0$. For $A,B,C\geq  0$, Young's inequality yields that
\begin{align*}
CA^aB^b&=\left(\tfrac{\delta A}{a}\right)^a\left(\tfrac{\varepsilon B}{b}\right)^bC\left(\tfrac{a}{\delta}\right)^a\left(\tfrac{b}{\varepsilon}\right)^b\\
&\leq a\tfrac{\delta A}{a}+b\tfrac{\varepsilon B}{b}+\gamma \left[C(\tfrac{a}{\delta})^a(\tfrac{b}{\varepsilon})^b\right]^{1/\gamma}\\
&= \delta A+\varepsilon B+C_{\varepsilon,\delta},
\end{align*}
where
$$
C_{\varepsilon,\delta}=\frac{Q^2}{4\varepsilon (Q+\alpha)^2}C^{\frac{2(Q+\alpha)}{Q}}\left(\frac{\alpha}{\delta(Q+\alpha)}\right)^{2\alpha/Q}.
$$
Applying this with $A=\|v\|_1$, $B=\E_\alpha(v)$, $C=C_{\G,\alpha}^{1/2}\|f\|_2$, $\varepsilon=1/4$, and $\delta=\lambda/2$, one concludes that
$$
 \int_\G fv\leq\frac14\E_\alpha(v)+\frac\lambda2\|v\|_1+C_0.
$$
Substituting this into \eqref{engy} gives \eqref{coeq}.
\end{proof}

We can now apply the direct method to study $\inf_{v\in \K_\alpha} \mathcal{J}_\lambda^\alpha(v)$.

\begin{proposition}\label{min}
There is a unique $u\in\K_\alpha$ such that
$$
 \mathcal J_\lambda^\alpha(u) =\inf_{v\in\K_\alpha}\mathcal J_\lambda^\alpha(v).
$$
It is characterized by the variational inequality
\begin{equation}\label{vi}
 \E_\alpha(u,v-u)
 \geq\int_\G(f-\lambda)(v-u),
 \qquad \text{for all } v\in\K_\alpha,
\end{equation}
and satisfies the following energy identity:
\begin{equation}\label{iden}
 \E_\alpha(u)=\int_\G(f-\lambda)u.
\end{equation}
\end{proposition}

\begin{proof}
		For any $v,w\in\K_\alpha$, a direct computation gives
	$$
	\mathcal J_\lambda^\alpha\left(\tfrac{v+w}{2}\right)=\tfrac12\mathcal J_\lambda^\alpha(v)+\tfrac12\mathcal J_\lambda^\alpha(w) -\tfrac18\E_\alpha(v-w)\leq \tfrac12\mathcal J_\lambda^\alpha(v)+\tfrac12\mathcal J_\lambda^\alpha(w).
	$$
	The equality holds if and only if $\E_\alpha(v-w)=0$, which implies $\Lop^{\alpha/4}(v - w) = 0$ and thus $v = w$. Hence $\mathcal{J}_\lambda^\alpha$ is strictly convex, which implies that the minimizer is unique if it exists.
    
    We now employ the direct method to prove the existence. Let $\{v_k\}_{k=1}^\infty$ be a minimizing sequence. Since $\inf_{v\in \K_\alpha}\mathcal{J}_{\lambda}^\alpha(v) \leq \mathcal{J}_{\lambda}^\alpha(0)=0$, after discarding finitely many terms and using Proposition \ref{coer}, we may assume
    $$
    1>\mathcal{J}_\lambda^\alpha (v_k) \geq\frac14\E_\alpha(v_k)+\frac\lambda2\|v_k\|_1-C_0.
    $$
    Thus both $\E_\alpha(v_k)$ and $\|v_k\|_1$ are uniformly bounded. Moreover, Lemma \ref{nash} ensures that the $L^2$ norms of $\{v_k\}_{k = 1}^\infty$ are also uniformly bounded. 

   Therefore, $\{v_k\}_{k = 1}^\infty$ is bounded in $(\mathsf{V}_\alpha,(\|v\|_2^2+\|\Lop^{\alpha/4}v\|_2^2)^{1/2})$.
   After passing to a subsequence, we may assume that
   $$
    v_k\rightharpoonup u \qquad\text{weakly in }\mathsf V_\alpha.
   $$
   Consequently, we have
	$$
	v_k\rightharpoonup u\qquad\text{and}\qquad \Lop^{\alpha/4}v_k\rightharpoonup \Lop^{\alpha/4}u, \qquad\text{ in $L^2(\G)$}.
	$$
    It follows immediately that
    \begin{equation} \label{ff1}
\E_\alpha(u)=\|\Lop^{\alpha/4}u\|_2^2\leq\liminf_{k \to \infty}\|\Lop^{\alpha/4}v_k\|_2^2=\liminf_{k \to \infty}\E_\alpha(v_k).
    \end{equation}
    For any $0\leq \phi\in L^2(\G)$, the weak convergence gives $\int_{\G} u\phi=\lim_{k \to \infty}\int_{\G} v_k\phi \geq 0$, which implies $u\geq0$ by taking $\phi=\max\{-u,0\}$. It follows then
    $$
    \|u\|_1=\sup_{\substack{\phi\in C_c^\infty(\G)\\0\leq\phi\leq1}}\int_\G u\phi
    =\sup_{\substack{\phi\in C_c^\infty(\G)\\0\leq\phi\leq1}}\lim_{k\to\infty}\int_\G v_k\phi
    \leq\liminf_{k\to\infty}\|v_k\|_1,
    $$
    which shows that $u\in \K_\alpha$.

The second inequality in \eqref{ff1} and the  $L^1$-norm estimate above,     together with the continuity of
    $v\mapsto\int_\G fv$ imply that
    \begin{align*}
     \inf_{v\in \K_\alpha}\mathcal J_{\lambda}^\alpha(v)\leq \mathcal J_\lambda^\alpha(u)&=\frac{1}{2}\E_\alpha(u)-\int_{\G}(f-\lambda)u\\
      & \leq  \liminf_{k\to\infty} \frac{1}{2}\E_\alpha(v_k) + \lambda  \liminf_{k\to\infty} \int_{\G}v_k - \lim_{k\to\infty} \int_{\G} f v_k\\
      &\leq \liminf_{k\to\infty} \left[\frac{1}{2}\E_\alpha(v_k)-\int_{\G}(f-\lambda)v_k\right]\\
     &=\liminf_{k\to\infty}\mathcal J_\lambda^\alpha(v_k)=\inf_{v\in \K_\alpha}\mathcal J_{\lambda}^\alpha(v).
    \end{align*}
    Thus $u$ is the unique minimizer.
	
    We now prove \eqref{vi}. Fix $v\in\K_\alpha$ and define
	$$
	u_t:=(1-t)u+tv\in\K_\alpha,\qquad 0\leq t\leq1.
	$$
	Expanding the quadratic form, we obtain
	\begin{align*}
	\psi(t):=\mathcal J_\lambda^\alpha(u_t)
		={}&\mathcal J_\lambda^\alpha(u)
		+t\left(
		\E_\alpha(u,v-u)-\int_\G(f-\lambda)(v-u)
		\right)
		+\frac{t^2}{2}\E_\alpha(v-u).
	\end{align*}
    Since $u$ minimizes $\mathcal J_\lambda^\alpha$, we have $\psi(t)\geq \psi(0)$ for $t\in[0,1]$.  Therefore,
    $$
    0\leq\psi'(0)=\E_{\alpha}(u,v-u)-\int_{\G} (f-\lambda)(v-u),
    $$
	which justifies \eqref{vi}. Conversely, suppose that $u\in\K_\alpha$ satisfies \eqref{vi}. Then for every $v\in\K_\alpha$,
	\begin{align*}
	\mathcal J_\lambda^\alpha(v)-\mathcal J_\lambda^\alpha(u)=\E_\alpha(u,v-u)
	-\int_\G(f-\lambda)(v-u)+\frac12\E_\alpha(v-u)\geq0.
	\end{align*}
	Thus \eqref{vi} is also a sufficient condition. 
    
    Finally, \eqref{iden} follows by taking $v=0$ and $v=2u$ in \eqref{vi} successively.
\end{proof}

\subsection{The Lewy--Stampacchia estimate}

Let $u$ be the minimizer given by Proposition \ref{min}. Define a distribution $\eta$ on $C_c^\infty(\G)$ by
\begin{equation}\label{eta}
 \langle\eta,\phi\rangle
 =\E_\alpha(u,\phi)-\int_\G(f-\lambda)\phi,
 \qquad\text{for all }\phi\in C_c^\infty(\G).
\end{equation}
Let $K$ be a fixed compact subset of $\G$. Then for every $\phi\in C_c^\infty(\G)$ with
$\operatorname{supp}\phi\subset K$, the Cauchy--Schwarz inequality and
\eqref{intr} imply
\begin{align*}
	|\E_\alpha(u,\phi)|&\leq \E_\alpha(u)^{1/2}\|\Lop^{\alpha/4}\phi\|_2\leq
	\E_\alpha(u)^{1/2} \|\phi\|_2^{1-\alpha/4}\|\Lop\phi\|_2^{\alpha/4}\\
	&\leq \E_\alpha(u)^{1/2} (|K|^{1/2}\|\phi\|_\infty)^{1-\alpha/4}(|K|^{1/2}\|\Lop \phi\|_\infty)^{\alpha/4}\\
	&\leq|K|^{1/2}\E_\alpha(u)^{1/2} [(1-\tfrac{\alpha}{4})\|\phi\|_\infty+\tfrac{\alpha}{4}\|\Lop \phi\|_\infty],
	\end{align*}
where we used Young's inequality in the last inequality. We also have
$$
\left|\int_\G(f-\lambda)\phi\right|\leq
\bigl(\|f\|_{L^1(K)}+\lambda|K|\bigr)\|\phi\|_\infty.
$$
These two estimates show that $\eta$ is a distribution. Moreover, for any $0\leq\phi\in C_c^\infty(\G)$,  taking $v=u+\phi$ in the variational inequality \eqref{vi} yields
$$
\langle \eta,\phi\rangle=\E_\alpha(u,\phi)-\int_\G(f-\lambda)\phi\geq 0.
$$
Therefore $\eta$ is a positive distribution. Recall that every positive distribution is represented by a unique positive Radon measure, see for instance \cite[Theorem 6.22]{LL01}. We will use the same symbol $\eta$ for the resulting positive Radon measure subsequently, so that
\begin{align*}
	\langle\eta,\phi\rangle=\int_\G\phi\,d\eta,\qquad \text{for every }\phi\in C_c^\infty(\G).
\end{align*}

Moreover, we have the following stronger proposition.
\begin{proposition}\label{lewy}
Define $\eta$ by \eqref{eta}. Then for every nonnegative $\phi\in C_c^\infty(\G)$,
\begin{equation}\label{lseq}
 0\leq\langle\eta,\phi\rangle\leq\int_\G(\lambda-f)^+\phi\leq\lambda\|\phi\|_1.
\end{equation}
Consequently, $\eta$ is absolutely continuous with respect to Haar measure. We henceforth use the same symbol $\eta$ for its Radon--Nikodym density.  With this convention,
$$
0\leq\eta\leq(\lambda-f)^+\leq\lambda,\qquad\text{almost everywhere}.
$$
\end{proposition}
\begin{proof}
    Fix a nonnegative function $\phi\in C_c^\infty(\G)$ and a number $\varepsilon>0$. Set
	$$
	z:=u-\varepsilon\phi,
	\qquad
	v:=z^+=\max\{z,0\}.
	$$
    Since the maps $s\mapsto s^\pm$ are $1$-Lipschitz and vanish at the origin, Lemma \ref{mark}(i) implies $z^\pm\in\mathsf V_\alpha$. Moreover, since $u\geq0$ and $\phi\geq0$,  we have
    $$
    0\leq z^+\leq u.
    $$
    Thus $v=z^+\in\K_\alpha$. 
    
    Applying \eqref{vi} with $v-u=z^- -\varepsilon\phi$ (since $u - \varepsilon \phi = z = z^+ - z^-$) and expanding both sides, we obtain
    $$
    \E_\alpha(u,z^-)-\varepsilon\E_\alpha(u,\phi)
    \geq\int_\G(f-\lambda)z^-
    -\varepsilon\int_\G(f-\lambda)\phi.
    $$
    That is, 
	\begin{equation}\label{epsi}
		\varepsilon\langle\eta,\phi\rangle\leq\E_\alpha(u,z^-)+\int_\G(\lambda-f)z^-.
	\end{equation}
	
	It remains to bound the two terms on the right-hand side. 
    Using $u=z+\varepsilon\phi$, Lemma \ref{mark}(iii) and the Cauchy--Schwarz inequality, we can estimate
	\begin{align*}
	\E_\alpha(u,z^-)&=\E_\alpha(z,z^-)+\varepsilon\E_\alpha(\phi,z^-)\leq -\E_\alpha(z^-) +\varepsilon \E_\alpha(\phi)^{1/2}\E_\alpha(z^-)^{1/2}\\
	&=-(\E_\alpha (z^-)^{1/2}-\tfrac{\varepsilon}{2}\E_\alpha(\phi)^{1/2})^2+\tfrac{\varepsilon^2}{4}\E_\alpha(\phi)\leq \tfrac{\varepsilon^2}{4}\E_\alpha(\phi).
	\end{align*}
    By definition, $0\leq z^-=(\varepsilon \phi -u)^+\leq \varepsilon \phi$ . Thus
	\begin{equation*}
		\int_\G(\lambda-f)z^-\leq\int_\G(\lambda-f)^+z^-\leq\varepsilon\int_\G(\lambda-f)^+\phi.
	\end{equation*}
	Inserting these two estimates into \eqref{epsi} and dividing both sides by $\varepsilon>0$, we see that
	$$
	\langle\eta,\phi\rangle\leq\frac{\varepsilon}{4}\E_\alpha(\phi)	+\int_\G(\lambda-f)^+\phi.
	$$
	Since $f\geq 0$, letting $\varepsilon \to 0^+$ yields
	$$
		\langle\eta,\phi\rangle\leq\int_\G(\lambda-f)^+\phi\leq \int_\G \lambda \phi=\lambda\|\phi\|_1.
	$$
	This proves \eqref{lseq} and implies that $\eta$ is absolutely continuous with respect to $dx$. Identifying $\eta$ with its Radon--Nikodym density, the inequality above gives
   $$
   0\leq\eta\leq(\lambda-f)^+\leq\lambda\qquad\text{almost everywhere}.
   $$
   The proof of Proposition \ref{lewy} is therefore finished.
\end{proof}

\subsection{Mass, regularity, and complementarity}
By Proposition \ref{lewy}, $\eta$ now denotes a bounded nonnegative
function.  We set
\begin{equation}\label{mudf}
	\mu:=\lambda-\eta.
\end{equation}
Using Proposition \ref{lewy} and \eqref{eta}, we know that
\begin{equation}\label{weak}
	\int_\G\mu\phi=\int_\G f\phi-\E_\alpha(u,\phi),\qquad\phi\in C_c^\infty(\G)
\end{equation}
and
\begin{equation}\label{claimi}
	0\leq\mu\leq\lambda,
	\qquad
	\mu\geq\lambda-(\lambda-f)^+=\min\{\lambda,f\}\quad \text{almost everywhere}.
\end{equation}
\begin{lemma}\label{mass}
	The function $\mu$ defined by \eqref{mudf} belongs to $L^1(\G)$ and satisfies
	\begin{equation*}
		\int_\G\mu=\int_\G f.
	\end{equation*}
\end{lemma}

\begin{proof}
	Let $\chi_R$ be the cutoff functions given in Lemma \ref{cut}. Recall that
	$$
	0\leq\chi_R\leq1,\quad\chi_R\to1 \quad\text{pointwise on }\G,\quad\text{and}\quad
	\|\Lop^{\alpha/2}\chi_R\|_\infty \leq C_{\chi,\alpha}R^{-\alpha}.
	$$
	We take $\phi=\chi_R$ in \eqref{weak}. Since $\chi_R\in D(\Lop^{\alpha/2})$ and $u\in D(\Lop^{\alpha/4})$,  we have 
	\begin{align*}
		\E_\alpha(u,\chi_R)=\langle\Lop^{\alpha/4}u,\Lop^{\alpha/4}\chi_R\rangle_2=\langle u,\Lop^{\alpha/4}\Lop^{\alpha/4}\chi_R\rangle_2=\langle u,\Lop^{\alpha/2}\chi_R\rangle_2.
	\end{align*}
	Therefore,
	\begin{equation*}
		|\E_\alpha(u,\chi_R)|
		\leq\|u\|_1\|\Lop^{\alpha/2}\chi_R\|_\infty
		\leq C_{\chi,\alpha}R^{-\alpha}\|u\|_1.
	\end{equation*}
	By the dominated convergence theorem, $\int f\chi_R\to \int f$ and hence
	$$
	\lim_{R\to\infty}\int_\G\mu\chi_R=
	\lim_{R\to\infty}\int_\G f\chi_R-\lim_{R\to\infty}\E_\alpha(u,\chi_R)
	=\int_\G f.
	$$
	
	Since $\mu\geq0$, Fatou's lemma now implies
	$$
	\int_\G\mu
	\leq\liminf_{R\to\infty}\int_\G\mu\chi_R
	=\int_\G f<\infty.
	$$
	Thus $\mu\in L^1(\G)$. Applying the dominated convergence theorem again, we have
	$$
	\int_\G\mu
	=\lim_{R\to\infty}\int_\G\mu\chi_R
	=\int_\G f,
	$$
    as required.
\end{proof}

We next upgrade the weak relation \eqref{weak} to an operator identity.
\begin{proposition}\label{reg}
The minimizer $u$ given by Proposition \ref{min} satisfies
$$
u\in D(\Lop^{\alpha/2}),\qquad
\Lop^{\alpha/2}u=f-\mu\quad\text{in }L^2(\G).
$$
Equivalently, $u$ satisfies
\begin{equation}\label{form}
 \E_\alpha(u,\psi)=\int_\G(f-\mu)\overline{\psi},
 \qquad \text{for every }\psi\in\mathsf V_\alpha.
\end{equation}
Finally, if $\alpha\geq1$, then $u\in D(\Lop^{1/2})=W_H^{1,2}(\G;\mathbb{C})$.
\end{proposition}

\begin{proof}
	By Lemma \ref{mass} and the bound $0\leq \mu \leq \lambda$, we have
	$$
	\mu\in L^1(\G)\cap L^\infty(\G)\subset L^2(\G).
	$$
	By the core property \eqref{core}, we can approximate any real
	$\psi\in\mathsf V_\alpha$ by functions $\{\phi_k\}_{k = 1}^\infty\subset C_c^\infty(\G)$. Therefore, by \eqref{eta} and \eqref{mudf} it holds that,
	\begin{align*}
	|\E_\alpha(u,\psi)-\langle f-\mu,{\psi}\rangle_2|&\leq  |\E_\alpha(u,\psi)-\E_\alpha(u,{\phi_k})|+|\langle f-\mu,{\phi_k}\rangle_2-\langle f-\mu,{\psi}\rangle_2|\\
	&\quad+|\E_\alpha(u,\phi_k)-\langle f-\mu,{\phi_k}\rangle_2|\\
	&\leq \E_\alpha(u)^{1/2}\E_\alpha(\phi_k-\psi)^{1/2}+\|f-\mu\|_2\|\phi_k-\psi\|_2.
	\end{align*}
Now	letting $k\to \infty$ gives
	$$
	\E_\alpha(u,\psi)=\int_{\G}(f-\mu)\psi,\qquad \text{for every real }\psi\in \mathsf{V}_\alpha,
	$$
 which leads to \eqref{form} by the conjugate linearity of $\E_\alpha$. By the representation theorem for Dirichlet forms, identity \eqref{form} implies 
	$$
	u\in D(\Lop^{\alpha/2}),\qquad\Lop^{\alpha/2}u=f-\mu
	\quad\text{in }L^2(\G).
	$$
	
	Suppose that $\alpha\geq1$. Using the spectral theorem and the elementary inequality $\lambda\leq1+\lambda^\alpha$ for $\lambda\geq0$, we obtain
	$$
	\|\Lop^{1/2}u\|_2^2
	=\int_{[0,\infty)}\lambda d\langle E_\lambda u,u\rangle_2\leq \int_{[0,\infty)}(1+\lambda^\alpha)d\langle E_\lambda u,u\rangle_2=\|u\|_2^2+\|\Lop^{\alpha/2}u\|_2^2<\infty.
	$$
	Thus $u\in D(\Lop^{1/2})=W_H^{1,2}(\G;\mathbb{C})$.
\end{proof}

The following complementarity relation identifies the role of the obstacle.
The reaction density $\eta$ can be nonzero only on the contact set
$\{u=0\}$. 

\begin{lemma}\label{comp}
	One has $\eta u=0$ almost everywhere. Consequently,
	$$
    \eta=0,\qquad\mu=\lambda,
	\qquad\text{almost everywhere on }\Omega:=\{x\in\G:u(x)>0\}.
    $$
\end{lemma}

\begin{proof}
	Taking $\psi=u$ in \eqref{form} gives
	\begin{equation*}
		\E_\alpha(u)=\int_\G(f-\mu)u.
	\end{equation*}
	On the other hand, the energy identity \eqref{iden} says that
	\begin{equation*}
		\E_\alpha(u)=\int_\G(f-\lambda)u.
	\end{equation*}
	Combining these two equalities and using $\eta=\lambda-\mu$ (recall \eqref{mudf}), we obtain
	$$
	0=\int_\G(\lambda-\mu)u=\int_\G\eta u.
	$$
	Since $\eta u\geq0$ and its integral is zero, it follows that
	$\eta u=0$ almost everywhere. Since $u>0$ a.e. on $\Omega$,  we have $\eta=0$ and hence $\mu=\lambda$ almost everywhere on $\Omega$.
\end{proof}

We are now in a position to complete the proof of the decomposition theorem.

\begin{proof}[Proof of Theorem \ref{dcmp}]
	Let $u$ be the unique minimizer in Proposition \ref{min} and  $\mu=\lambda-\eta$. From the construction they are nonnegative. By Proposition \ref{reg}, it holds that
	$$
	f=\mu+\Lop^{\alpha/2}u\qquad\text{in }L^2(\G).
	$$
	Then assertion (i) is given by \eqref{claimi} and Lemma \ref{mass}.
	
	By Lemma \ref{comp}, $\mu=\lambda$ almost everywhere on
	$\Omega=\{u>0\}$, which proves assertion (ii).  Assertion (iii) follows from (i) and (ii) immediately since 
	$$
	\lambda|\Omega|=\int_\Omega\mu\leq\int_\G\mu
	=\int_\G f=\|f\|_{1}.
	$$
	
	It remains to prove the final statement. When $\alpha\geq1$, $u\in D(\Lop^{1/2})=W_{H}^{1,2}(\G;\mathbb{C})$ by Proposition \ref{reg}. 
	Moreover, by using Cauchy--Schwarz inequality,
	$u\in W_{H,\mathrm{loc}}^{1,1}(\G)$ (noticing that $u$ is real-valued).  Since $u=u^+ = \max\{u,0\}$, applying Lemma \ref{stam} with $c=0$ gives
	$$
	\gradH u=\gradH u^+=\1_{\{u>0\}}\gradH u.
	$$
	Thus $\gradH u=0$ almost everywhere on $\G\setminus\Omega$.
\end{proof}

\section{Proof of Theorem~\ref{thmcce}}\label{sce}

In this section we give the proof of Theorem~\ref{thmcce}. The idea of the proof is the following: we first modify the construction given in \cite[Counterexample 1 in \S~4]{D26} for Riesz transforms associated with Schr\"odinger operators; then an approximation of the function $V^\frac{1}{2}$ by polynomials implies the convergence of Dirichlet forms (in the sense of Mosco), which gives a sequence of operators with unbounded $L^p$ bounds of Riesz transforms associated with Schr\"odinger operators; finally we transfer the result to stratified Lie groups by using the results of \cite{RS16}. To realize the final step we first consider the complex-valued spaces instead of the real-valued ones.

\subsection{Riesz transforms associated with Schr\"odinger operators}

For the construction we will use the known results on Riesz transforms associated with Schr\"odinger operators. More precisely, we will modify the construction in \cite[\S~7]{S95} or \cite[Counterexample 1 in \S~4]{D26}: on $\mathbb{R}^3$, for any given $p > 2$, we pick an $\varepsilon \in (0,1 - 2/p)$ and define
\begin{align}\label{defvV}
    V(x,y,z) := (x^2 + y^2)^\frac{\varepsilon - 2}{2} \in L^1_{\mathrm{loc}}(\mathbb{R}^3), \qquad 
    v(x,y,z) := \sum_{k = 0}^\infty \frac{(x^2 + y^2)^\frac{\varepsilon k }{2}}{\varepsilon^{2k} (k!)^2}.
\end{align}
Throughout this subsection we use $\Delta := - ( \partial_x^2 + \partial_y^2 + \partial_z^2)$ and $\nabla$ to denote the Laplacian and gradient on the Euclidean $\mathbb{R}^3$ space respectively. It is direct to check that $(\Delta + V ) v = 0 $ in the distributional sense. 
Set $u(x,y,z) := \phi(\sqrt{x^2 + y^2}) \phi(z) v(x,y,z)$ with $\phi \in C_c^\infty(\mathbb{R})$ an even function which is $1$ near $0$ and  $g := (\Delta + V + 1) u$. A direct calculation gives
\[
g = \Delta (\phi(\sqrt{x^2 + y^2}) \phi(z)) v - 2\nabla (\phi(\sqrt{x^2 + y^2}) \phi(z)) \cdot \nabla v  + u,
\]
which implies $g$ is bounded and has compact support. 

\begin{lemma}\label{lUNbSR}
For every fixed $p > 2$ and $\varepsilon \in (0,1 - 2/p)$, define $V$ in \eqref{defvV}. The operator $\partial_x (\Delta + V + 1)^{-1/2}$ is not bounded on $L^p(\mathbb{R}^3;\mathbb{C})$.
\end{lemma}

\begin{proof}
Note that the following sesquilinear   form is a Dirichlet form:
\begin{align}\label{deftE}
    \tilde{\E}(v,w) :=  \langle \nabla v, \nabla w \rangle_2 + \langle V v,w \rangle_2.
\end{align}
Recall that $\langle f, g \rangle_2 = \int_{\mathbb{R}^3} f \bar{g} $ and $D(\tilde{\E}) := \{v \in L^2(\mathbb{R}^3;\mathbb{C}):  \tilde{\E}(v) :=  \tilde{\E}(v,v) < \infty\}$. In fact, it is easy to prove that it is closed and  if $v \in D(\tilde{\E})$, $\Phi:\mathbb C\to\mathbb C$ is $1$-Lipschitz and $\Phi(0)=0$, then it follows from definition that $v \in W^{1,2}(\mathbb{R}^3;\mathbb{C})$ and by the classical mapping property of the Sobolev spaces (see \cite[Theorem 2.1.11]{Z89}) we have 
$\tilde{\E}(\Phi\circ v)\leq\tilde{\E}(v)$ (see also the argument in the proof of Lemma \ref{mark} by using \eqref{jump22p}). If we use $Q_t$ to denote its associated semigroup, then we know $Q_t$ is a contraction on $L^p(\mathbb{R}^3;\mathbb{C})$ (cf. the argument in \cite[the argument in the proof of Theorem 1.3.3]{D89} together with \cite[Theorem 1.4.1]{FOT11}). Now by the formula:
\begin{align}\label{repres}
(\Delta + V + 1)^{-1/2}f
  =\frac1{\sqrt\pi}\int_0^\infty
       t^{-1/2}e^{-t} \, Q_tf dt,
\end{align}
we obtain that $(\Delta + V + 1)^{-1/2}$ is also bounded in $L^p(\mathbb{R}^3;\mathbb{C})$. Then we argue by contradiction. Assume that the operator $\partial_x (\Delta + V + 1)^{-1/2}$ is also bounded on $L^p(\mathbb{R}^3;\mathbb{C})$, then we have the operator $\partial_x (\Delta + V + 1)^{-1}$ is also bounded on $L^p(\mathbb{R}^3;\mathbb{C})$. However, the function $g$ constructed above has finite $L^p$ norm but 
\[
\partial_x (\Delta + V + 1)^{-1} g = \partial_x u
\]
does not belong to $L^p(\mathbb{R}^3;\mathbb{C})$ since near the origin we have 
\[
|\partial_x u| \sim |x| (x^2 + y^2)^{\varepsilon/2 - 1},
\]
where $A \sim B$ means there exists a constant $C > 0$ such that $C^{-1} A \le B \le C A$. This leads to a contradiction.
\end{proof}

\subsection{Mosco convergence and polynomial approximation}\label{ssMo}

Given a measure space $X = (X,\nu)$, for any Dirichlet form $\E$ on $L^2(X;\mathbb{C})$, we extend the definition of $\E$ on $D(\E) := \{u \in L^2(X;\mathbb{C}): \E(u) := \E(u,u) < \infty\}$ to the whole $L^2(X;\mathbb{C})$ by letting $\E(u) :=\infty$ for $u \notin D(\E)$. Then we have the notion of convergence of Dirichlet forms, which is called Mosco convergence in the literature (see for example \cite{M94}).

\begin{definition}[Mosco convergence]\label{defMo}
We say a Dirichlet form $\E$ on $L^2(X;\mathbb{C})$ is the Mosco limit of a sequence of Dirichlet forms $\{\E_k\}_{k = 1}^\infty$ on $L^2(X;\mathbb{C})$ if the following conditions hold:
\begin{enumerate}[label=\textup{(\roman*)}]
    \item for every $u_k \rightharpoonup u$ (converges weakly) in $L^2(X;\mathbb{C})$, we have
    \[
    \E(u) \le \liminf_{k \to \infty}\E_k(u_k);
    \]
    \item for every $u \in L^2(X;\mathbb{C})$, there exists a sequence $\{u_k\}_{k = 1}^\infty$ that converges to $u$ strongly  in $L^2(X;\mathbb{C})$ such that
    \[
     \E(u) \ge \limsup_{k \to \infty}\E_k(u_k).
    \]
\end{enumerate}
\end{definition}

Now we construct our $\E$ and $\{\E_k\}_{k = 1}^\infty$. To this end, with the function $V$ in \eqref{defvV}, let 
\[
W_k(x,y,z) := \min\{k,V(x,y,z)^\frac{1}{2}\},
\]
which is continuous on $\mathbb{R}^3$. Let $B_k := \{(x,y,z) \in \mathbb{R}^3: x^2 + y^2 + z^2 \le k^2\}$. Then on $B_k$ we can find a real polynomial $P_k$ such that
\begin{align}\label{approxP}
    |P_k - W_k| \le \frac{1}{k}, \qquad \mbox{on} \quad B_k.
\end{align}
From the choice of $P_k$ we know that for almost every point in $\mathbb{R}^3$ we have $P_k^2 \to V$ and the following estimate holds:
\begin{align}\label{approxP2}
    |P_k^2 - W_k^2| \le |P_k - W_k||P_k - W_k + 2 W_k| \le \frac{1}{k}\left( \frac{1}{k} + 2k \right) \le 3, \qquad \mbox{on} \quad B_k，
\end{align}
which implies
\begin{align}\label{approxP3}
0 \le P_k^2 =   |P_k^2| \le 3 + V, \qquad \mbox{on} \quad B_k.
\end{align}

\begin{remark}\label{redP}
If we use $\deg P_k$ to denote the degree of $P_k$ as a polynomial, then we must have $\deg P_k \to \infty$ as $k \to \infty$. In fact, if we define $\mathfrak p_k (x) := P_k(x,0,0)$, we must have $\deg \mathfrak p_k \le \deg P_k$. So it is sufficient to prove that $\deg \mathfrak p_k \to \infty$ as $k \to \infty$. Now for any $k \ge 3$, it follows from \eqref{approxP} that
\[
\mathfrak p_k(0) \ge k - \frac{1}{k} \ge \frac{8k}{9}, \qquad \mathfrak p_k\left( (k/2)^{\frac{2}{\varepsilon - 2}} \right) \le \frac{k}{2} + \frac{1}{k} \le \frac{11k}{18}.
\]
By the mean value theorem, there exists a $\xi \in [0,(k/2)^{2/(\varepsilon - 2)}] \subset [0,1]$ such that
\[
\mathfrak p_k'(\xi) = \frac{\mathfrak p_k\left( (k/2)^{\frac{2}{\varepsilon - 2}} \right) - \mathfrak p_k(0)}{(k/2)^{2/(\varepsilon - 2)} - 0} \le - \frac{5k}{18}   (k/2)^{-2/(\varepsilon - 2)}.
\]
Then Markov's inequality (see \cite[Theorem 5.1.8]{BE95}) gives 
\[
\frac{5k}{18}   (k/2)^{-2/(\varepsilon - 2)} \le \max_{[-1,1]} |\mathfrak p_k'| \le (\deg p_k)^2 \max_{[-1,1]} |\mathfrak p_k| \le (\deg p_k)^2 \frac{10k}{9},
\]
which implies 
\[
\deg p_k \ge \frac{1}{2}  (k/2)^{1/( 2 -\varepsilon)} \to \infty
\]
as $k \to \infty$.
\end{remark}

Now we define 
\begin{align}\label{defEE}
    \E(v,w) :=&  \langle \nabla v, \nabla w \rangle_2 + \langle (V + 1) v,w \rangle_2, \\
    \label{defEEk}
     \E_k(v,w) :=&  \langle \nabla v, \nabla w \rangle_2 + \langle (P_k^2 + 1) v,w \rangle_2.
\end{align}
With the same argument as the one below \eqref{deftE}, we know $\E$ and $\{\E_k\}_{k = 1}^\infty$ are Dirichlet forms on $L^2(\mathbb{R}^3;\mathbb{C})$. 

\begin{lemma}\label{lMoc}
With Dirichlet forms $\E$ and $\{\E_k\}_{k = 1}^\infty$ defined in \eqref{defEE} and \eqref{defEEk} respectively, $\E$ is the Mosco limit of $\{\E_k\}_{k = 1}^\infty$.
\end{lemma}

\begin{proof}
We begin with the first item. Without loss of generality we can assume that 
\[
\lim_{k \to \infty} \E_k(u_k) \qquad \mbox{has a finite limit},
\]
otherwise we can pick a subsequence to realize the lower limit and there is nothing to prove if the limit is $\infty$. From the definition \eqref{defEEk} we know that $u_k$ is bounded in $W^{1,2}(\mathbb{R}^3;\mathbb{C})$. Since the embedding from $W^{1,2}(\mathbb{R}^3;\mathbb{C})$ to $L^2(B_\ell; \mathbb{C})$ is compact (see for example \cite[Theorem 6.3]{AF03}), we can pick a subsequence (still call $u_k$) such that $u_k$ converges to $u$ almost everywhere on $B_\ell$. By a diagonal argument we can pick a subsequence (still call $u_k$) such that $u_k$ converges to $u$ almost everywhere on $\mathbb{R}^3$.  Pick a subsequence again if necessary we can make $u_k \rightharpoonup u$ also in  $W^{1,2}(\mathbb{R}^3;\mathbb{C})$. Lower semicontinuity of the $W^{1,2}(\mathbb{R}^3;\mathbb{C})$ norm gives
\[
\langle \nabla u, \nabla u \rangle_2 + \langle u,u \rangle_2 \le \liminf_{k \to \infty} (\langle \nabla u_k, \nabla u_k \rangle_2 + \langle u_k,u_k \rangle_2).
\]
On the other hand, Fatou's lemma gives
\[
\langle V u,u \rangle_2 \le \liminf_{k \to \infty} \langle P_k^2 u_k,u_k \rangle_2.
\]
Combining these two inequalities together, we obtain
\[
\E(u) \le  \liminf_{k \to \infty} (\langle \nabla u_k, \nabla u_k \rangle_2 + \langle u_k,u_k \rangle_2) + \liminf_{k \to \infty} \langle P_k^2 u_k,u_k \rangle_2 \le \liminf_{k \to \infty} \E_k(u_k).
\]

We are left to prove the second item. To this end, we fix a cutoff function $\chi \in C_c^\infty(\mathbb{R}^3)$ such that $0 \le \chi \le 1$, $\chi = 1$ on $B_{1/2}$ and $\chi = 0$ outside $B_1$. Let $\chi_k := \chi(\cdot/k)$. For every  $u \in L^2(\mathbb{R}^3;\mathbb{C})$, we can pick $u_k := \chi_k u$. It is clear that $u_k \to u$ in $L^2(\mathbb{R}^3;\mathbb{C})$.  Without loss of generality we can assume $\E(u) < \infty$, otherwise there is nothing to prove.  Then \eqref{approxP3}, together with the dominated convergence theorem, gives 
\[
\lim_{k \to \infty} \langle (P_k^2 + 1) u_k,u_k \rangle_2 = \langle (V + 1) u,u \rangle_2.
\]
Notice that 
\[
\nabla u_k  - \nabla u = u \nabla \chi_k  + (\chi_k - 1) \nabla u.
\]
It is not hard to prove that $\nabla u_k$ converges to $\nabla u$ in $L^2$-norm. Combining these two convergences above, we obtain
\[
     \E(u) = \lim_{k \to \infty}\E_k(u_k).
    \]
\end{proof}

As a corollary, we have the following result. 

\begin{corollary}\label{cstrc}
The operator $(\Delta + P_k^2 + 1)^{-1/2}$ converges to $(\Delta + V + 1)^{-1/2}$ in the strong operator topology of $L^2(\mathbb{R}^3;\mathbb{C})$.
\end{corollary}

\begin{proof}
In the proof we define the Dirichlet form
\begin{align}\label{deftEk}
    \tilde{\E}_k(v,w) :=  \langle \nabla v, \nabla w \rangle_2 + \langle P_k^2 v,w \rangle_2
\end{align}
and we denote its semigroup by $(Q_k)_t$, which is a contraction on $L^2(\mathbb{R}^3;\mathbb{C})$. Recall that $\tilde{\E}$ is given by \eqref{deftE} and $Q_t$ is its semigroup. From our construction \eqref{defEE} and \eqref{defEEk} we know their corresponding semigroups are $e^{-t} \, Q_t$ and $e^{-t} \, (Q_k)_t$ respectively. By using \eqref{repres}, for every $f \in L^2(\mathbb{R}^3;\mathbb{C})$, we have
\begin{align}\label{repres2}
\| (\Delta + P_k^2 + 1)^{-1/2}f - (\Delta + V + 1)^{-1/2}f \|_{2}
  \le \frac1{\sqrt\pi}\int_0^\infty
       t^{-1/2}  e^{-t} \|  (Q_k)_t f-   Q_tf  \|_2dt,
\end{align}
where we used the Minkowski inequality in the last inequality. Then for every $\epsilon > 0$, we first pick $T = T(\epsilon) > 0$ such that
\[
\frac1{\sqrt\pi}\int_T^\infty t^{-1/2}  e^{-t} dt \le \epsilon.
\]
It follows from Lemma \ref{lMoc} and \cite[Corollary 2.6.1]{M94} that the semigroup $e^{-t} \, (Q_k)_t$ converges to $e^{-t} \, Q_t$ in the strong operator topology of $L^2(\mathbb{R}^3;\mathbb{C})$ uniformly on $(0,T]$. As a result, we can pick $k_0$ such that for all $k \ge k_0$, we have 
\[
 \|e^{-t}\,  (Q_k)_t f-  e^{-t}Q_tf  \|_2 \le \frac{\epsilon \sqrt{\pi}}{\sqrt{T}},\qquad \forall \, t \in (0, T].
\]
Then we split the integral in \eqref{repres2}:
\begin{align*}
& \| (\Delta + P_k^2 + 1)^{-1/2}f - (\Delta + V + 1)^{-1/2}f \|_{2} \\
  \le \, &     \frac1{\sqrt\pi}\int_0^T
       t^{-1/2}  \|e^{-t}\,  (Q_k)_t f-  e^{-t}Q_tf  \|_2 dt + \frac1{\sqrt\pi}\int_T^\infty
       t^{-1/2}  e^{-t} \|  (Q_k)_t f-   Q_tf  \|_2dt
       \le  \epsilon +  2 \epsilon \|f\|_2,
\end{align*}
where we have used the fact that both $(Q_k)_t$ and $Q_t$ are contractions. 
\end{proof}

\begin{proposition}\label{pkkey}
For every fixed $p > 2$ and $\varepsilon \in (0,1 - 2/p)$, define $V$ in \eqref{defvV} and $P_k$ in \eqref{approxP}. Then we have
\[
\sup_{k}\|\partial_x (\Delta + P_k^2 + 1)^{-1/2}\|_{L^p(\mathbb{R}^3;\mathbb{C}) \to L^p(\mathbb{R}^3;\mathbb{C})} = \infty.
\]
\end{proposition}

\begin{proof}
Assume to the contrary that there exists a constant $C_p' > 0$ such that 
\begin{align}\label{conPk}
    \|\partial_x (\Delta + P_k^2 + 1)^{-1/2} f\|_p \le C_p' \|f\|_p,\qquad \forall \, k.
\end{align}
Now fix $f \in L^p(\mathbb{R}^3;\mathbb{C}) \cap L^2(\mathbb{R}^3;\mathbb{C})$, we have $\partial_x (\Delta + P_k^2 + 1)^{-1/2} f$ is bounded in $L^p(\mathbb{R}^3;\mathbb{C})$, which implies (passing to a subsequence if necessary) it converges to some function $g \in L^p(\mathbb{R}^3;\mathbb{C})$ weakly in $L^p(\mathbb{R}^3;\mathbb{C})$ and thus in the distributional sense. Corollary \ref{cstrc} implies that $\partial_x (\Delta + P_k^2 + 1)^{-1/2} f$ converges to $\partial_x (\Delta + V + 1)^{-1/2}f $ in the distributional sense so $g = \partial_x (\Delta + V + 1)^{-1/2}f$. Then by the lower semicontinuity of the $L^p(\mathbb{R}^3;\mathbb{C})$ norm and \eqref{conPk} that 
\[
\|\partial_x (\Delta + V + 1)^{-1/2} f\|_p \le \liminf_{k \to \infty} \|\partial_x (\Delta + P_k^2 + 1)^{-1/2} f\|_p \le C_p' \|f\|_p
\]
holds for all $f \in L^p(\mathbb{R}^3;\mathbb{C}) \cap L^2(\mathbb{R}^3;\mathbb{C})$ thus it can be extended to all $f \in L^p(\mathbb{R}^3;\mathbb{C})$, which leads to a contradiction with Lemma \ref{lUNbSR}.
\end{proof}

\subsection{Going back to Riesz transforms on groups}
\begin{proof}[Proof of Theorem \ref{thmcce}]
Fix $p > 2$. Pick an $\varepsilon \in (0,1 - 2/p)$ and define $V$ in \eqref{defvV} and $P_k$ in \eqref{approxP}.
Now for every $k$ we have the real polynomial $P_k$ on $\mathbb{R}^3$ and we can define the following operators on $L^p(\mathbb{R}^3;\mathbb{C})$
\begin{equation}\label{defAk}
  A_j^k : =\partial_{x_j}\quad (j=1,2,3),
  \qquad A_4^k : =iP_k,
  \qquad A_5^k : =iI,
\end{equation}
where $I$ is the identity operator.
Note that for $j = 1,2,3$, the operator $A_j^k$ is the generator of the translation and for $j = 4,5$, the operator $A_4^k$ and $A_5^k$ are the generator of the multiplication $e^{itP_k}$ and $e^{it}$. As a consequence, their one-parameter groups act isometrically on $L^p(\mathbb{R}^3;\mathbb{C})$ for all $p \in [1,\infty]$. Let $G_k$ be the Lie group of the real Lie algebra generated by $A_k^j, 1 \le j \le 5$. Then we know $G_k$ acts isometrically on $L^p(\mathbb{R}^3;\mathbb{C})$ for all $p \in [1,\infty]$ as well and it is nilpotent. Assume the step of $G_k$ is $s_k$. Let $\G_k$ be the free  Carnot group of  step $s_k$ generated by $A_k^j, 1 \le j \le 5$. By the universal property of free  Carnot groups (see \cite[Definition 14.1.1]{BLU07}), there is a group homomorphism from $\G_k$ to $G_k$.

Recall that the heat semigroup is a contraction semigroup, it follows from \cite[Theorem 2.3]{RS16} (with $a_p\le 2$ and $k = 0$ there) and Proposition \ref{pkkey} that
\[
\sup_{k}\|\Rvec^{\G_k}\|_{L^p(\G_k;\mathbb{C}) \to L^p(\G_k;\mathbb{C}^5)} \ge \frac12 \sup_{k}\|\partial_x (\Delta + P_k^2 + 1)^{-1/2}\|_{L^p(\mathbb{R}^3;\mathbb{C}) \to L^p(\mathbb{R}^3;\mathbb{C})} = \infty.
\]
Since we have 
\[
\|\Rvec^{\G_k}\|_{L^p(\G_k;\mathbb{C})  \to L^p(\G_k;\mathbb{C}^5)} \le 2 \|\Rvec^{\G_k}\|_{L^p(\G_k;\mathbb{R})  \to L^p(\G_k;\mathbb{R}^5)},
\]
we know also 
\[
\sup_{k} \|\Rvec^{\G_k}\|_{L^p(\G_k;\mathbb{R})  \to L^p(\G_k;\mathbb{R}^5)} = \infty.
\]
Picking a subsequence if necessary we prove the theorem.
\end{proof}

\begin{remark}\label{restep}
The step $s_k$ in the proof above is determined by $\deg P_k$. In fact, according to Remark \ref{redP}, we have
\[
s_k = \deg P_k + 1 \to \infty
\]
as $k \to \infty$. 
\end{remark}

\paragraph*{Acknowledgements.}

S.-C. Mao is partially supported by the China Postdoctoral Science Foundation (Grant No.~2026M793367). Y. Wang acknowledges support from
the China Scholarship Council and Ludwig-Maximilians-Universität
München. Y. Zhang has received funding from the European
Research Council (ERC) under the European Union's Horizon 2020 research and innovation programme (grant agreement GEOSUB, No.~945655).

\paragraph*{Declaration of AI Use.}

During the preparation of this manuscript, the authors used OpenAI GPT-5.6 as a brainstorming tool to explore possible approaches and generate preliminary research ideas, and for English language editing and stylistic improvements, including grammar, wording, and readability. The authors reviewed and edited all outputs, independently verified all mathematical arguments and references, and take full responsibility for the content of the publication.

\printbibliography

@article {MWZ26,
    AUTHOR = {Mao, Sheng-Chen and Wang, Yaojun    and Zhang, Ye},
     TITLE = {Dimension-free $L^p$-bounds for the Riesz transform on
{H}-type groups},
   JOURNAL = {In preparation},
      YEAR = {2026},
}

@article {LF04,
    AUTHOR = {Lust-Piquard, Fran\c coise},
     TITLE = {Dimension free estimates for discrete {R}iesz transforms on
              products of abelian groups},
   JOURNAL = {Adv. Math.},
  FJOURNAL = {Advances in Mathematics},
    VOLUME = {185},
      YEAR = {2004},
    NUMBER = {2},
     PAGES = {289--327},
      ISSN = {0001-8708,1090-2082},
   MRCLASS = {42B15 (43A15 43A32 46L55 47A20)},
  MRNUMBER = {2060471},
MRREVIEWER = {Hong-Quan\ Li},
       DOI = {10.1016/j.aim.2003.07.002},
       URL = {https://doi.org/10.1016/j.aim.2003.07.002},
}

@article {DOY06,
    AUTHOR = {Duong, Xuan Thinh and Ouhabaz, El Maati and Yan, Lixin},
     TITLE = {Endpoint estimates for {R}iesz transforms of magnetic
              {S}chr\"odinger operators},
   JOURNAL = {Ark. Mat.},
  FJOURNAL = {Arkiv f\"or Matematik},
    VOLUME = {44},
      YEAR = {2006},
    NUMBER = {2},
     PAGES = {261--275},
      ISSN = {0004-2080,1871-2487},
   MRCLASS = {35Q40 (31B10 42B20 47F05 81Q10)},
  MRNUMBER = {2292721},
MRREVIEWER = {Andrei\ B.\ Bogatyr\"ev},
       DOI = {10.1007/s11512-006-0021-x},
       URL = {https://doi.org/10.1007/s11512-006-0021-x},
}

@article {CD99,
    AUTHOR = {Coulhon, Thierry and Duong, Xuan Thinh},
     TITLE = {Riesz transforms for {$1\leq p\leq 2$}},
   JOURNAL = {Trans. Amer. Math. Soc.},
  FJOURNAL = {Transactions of the American Mathematical Society},
    VOLUME = {351},
      YEAR = {1999},
    NUMBER = {3},
     PAGES = {1151--1169},
      ISSN = {0002-9947,1088-6850},
   MRCLASS = {58G11 (31C12 53C21 60J35)},
  MRNUMBER = {1458299},
MRREVIEWER = {Zhongmin\ Qian},
       DOI = {10.1090/S0002-9947-99-02090-5},
       URL = {https://doi.org/10.1090/S0002-9947-99-02090-5},
}

@book {BE95,
    AUTHOR = {Borwein, Peter and Erd\'elyi, Tam\'as},
     TITLE = {Polynomials and polynomial inequalities},
    SERIES = {Graduate Texts in Mathematics},
    VOLUME = {161},
 PUBLISHER = {Springer-Verlag, New York},
      YEAR = {1995},
     PAGES = {x+480},
      ISBN = {0-387-94509-1},
   MRCLASS = {41-02 (00A07 12-02 26-02 30-02 30C15 41A17)},
  MRNUMBER = {1367960},
MRREVIEWER = {D.\ S.\ Lubinsky},
       DOI = {10.1007/978-1-4612-0793-1},
       URL = {https://doi.org/10.1007/978-1-4612-0793-1},
}

@book {
BLU07,
    AUTHOR = {Bonfiglioli, A. and Lanconelli, E. and Uguzzoni, F.},
     TITLE = {Stratified {L}ie groups and potential theory for their
              sub-{L}aplacians},
    SERIES = {Springer Monographs in Mathematics},
 PUBLISHER = {Springer, Berlin},
      YEAR = {2007},
     PAGES = {xxvi+800},
      ISBN = {978-3-540-71896-3; 3-540-71896-6},
   MRCLASS = {22E30 (31C45 35-02 35H10 43A80)},
  MRNUMBER = {2363343},
MRREVIEWER = {Maria Stella Fanciullo},
}

@article {M94,
    AUTHOR = {Mosco, Umberto},
     TITLE = {Composite media and asymptotic {D}irichlet forms},
   JOURNAL = {J. Funct. Anal.},
  FJOURNAL = {Journal of Functional Analysis},
    VOLUME = {123},
      YEAR = {1994},
    NUMBER = {2},
     PAGES = {368--421},
      ISSN = {0022-1236,1096-0783},
   MRCLASS = {47N20 (31C25 35B27 46N20 47A99 60J99 73B27)},
  MRNUMBER = {1283033},
MRREVIEWER = {Yevgen\ Y.\ Khruslov},
       DOI = {10.1006/jfan.1994.1093},
       URL = {https://doi.org/10.1006/jfan.1994.1093},
}

@article {RS16,
    AUTHOR = {Robinson, Derek W. and Sikora, Adam},
     TITLE = {Gru\v sin operators, {R}iesz transforms and nilpotent {L}ie
              groups},
   JOURNAL = {Math. Z.},
  FJOURNAL = {Mathematische Zeitschrift},
    VOLUME = {282},
      YEAR = {2016},
    NUMBER = {1-2},
     PAGES = {461--472},
      ISSN = {0025-5874,1432-1823},
   MRCLASS = {35J70 (22E25 35A30 35H20 43A65)},
  MRNUMBER = {3448390},
       DOI = {10.1007/s00209-015-1548-y},
       URL = {https://doi.org/10.1007/s00209-015-1548-y},
}

@book {Z89,
    AUTHOR = {Ziemer, William P.},
     TITLE = {Weakly differentiable functions},
    SERIES = {Graduate Texts in Mathematics},
    VOLUME = {120},
      NOTE = {Sobolev spaces and functions of bounded variation},
 PUBLISHER = {Springer-Verlag, New York},
      YEAR = {1989},
     PAGES = {xvi+308},
      ISBN = {0-387-97017-7},
   MRCLASS = {46E35},
  MRNUMBER = {1014685},
MRREVIEWER = {V.\ M.\ Gol\cprime dshte\u in},
       DOI = {10.1007/978-1-4612-1015-3},
       URL = {https://doi.org/10.1007/978-1-4612-1015-3},
}

@book {D89,
    AUTHOR = {Davies, E. B.},
     TITLE = {Heat kernels and spectral theory},
    SERIES = {Cambridge Tracts in Mathematics},
    VOLUME = {92},
 PUBLISHER = {Cambridge University Press, Cambridge},
      YEAR = {1989},
     PAGES = {x+197},
      ISBN = {0-521-36136-2},
   MRCLASS = {35P15 (35J25 35P20 47F05 58G05 58G11 58G25)},
  MRNUMBER = {990239},
MRREVIEWER = {H.\ Triebel},
       DOI = {10.1017/CBO9780511566158},
       URL = {https://doi.org/10.1017/CBO9780511566158},
}

@book {AF03,
    AUTHOR = {Adams, Robert A. and Fournier, John J. F.},
     TITLE = {Sobolev spaces},
    SERIES = {Pure and Applied Mathematics (Amsterdam)},
    VOLUME = {140},
   EDITION = {Second},
 PUBLISHER = {Elsevier/Academic Press, Amsterdam},
      YEAR = {2003},
     PAGES = {xiv+305},
      ISBN = {0-12-044143-8},
   MRCLASS = {46E35 (46-01 46-02 46B70 46Exx)},
  MRNUMBER = {2424078},
}

@book {FOT11,
    AUTHOR = {Fukushima, Masatoshi and Oshima, Yoichi and Takeda, Masayoshi},
     TITLE = {Dirichlet forms and symmetric {M}arkov processes},
    SERIES = {De Gruyter Studies in Mathematics},
    VOLUME = {19},
   EDITION = {extended},
 PUBLISHER = {Walter de Gruyter \& Co., Berlin},
      YEAR = {2011},
     PAGES = {x+489},
      ISBN = {978-3-11-021808-4},
   MRCLASS = {60J25 (28A12 31C45 60F10 60J40 60J45 60J55)},
  MRNUMBER = {2778606},
}

@article {D26,
    AUTHOR = {Dziuba\'nski, Jacek},
     TITLE = {On dimension-free and potential-free estimates for {R}iesz
              transforms associated with {S}chr\"odinger operators},
   JOURNAL = {Nonlinear Anal.},
  FJOURNAL = {Nonlinear Analysis. Theory, Methods \& Applications. An
              International Multidisciplinary Journal},
    VOLUME = {262},
      YEAR = {2026},
     PAGES = {Paper No. 113918, 7},
      ISSN = {0362-546X,1873-5215},
   MRCLASS = {42B20 (35J10 47D08 47G10)},
  MRNUMBER = {4946110},
       DOI = {10.1016/j.na.2025.113918},
       URL = {https://doi.org/10.1016/j.na.2025.113918},
}

@article {S95,
    AUTHOR = {Shen, Zhong Wei},
     TITLE = {{$L^p$} estimates for {S}chr\"odinger operators with certain
              potentials},
   JOURNAL = {Ann. Inst. Fourier (Grenoble)},
  FJOURNAL = {Universit\'e{} de Grenoble. Annales de l'Institut Fourier},
    VOLUME = {45},
      YEAR = {1995},
    NUMBER = {2},
     PAGES = {513--546},
      ISSN = {0373-0956,1777-5310},
   MRCLASS = {35J10 (42B20)},
  MRNUMBER = {1343560},
MRREVIEWER = {V.\ S.\ Rabinovich},
       DOI = {10.5802/aif.1463},
       URL = {https://doi.org/10.5802/aif.1463},
}

@book {LL01,
    AUTHOR = {Lieb, Elliott H. and Loss, Michael},
     TITLE = {Analysis},
    SERIES = {Graduate Studies in Mathematics},
    VOLUME = {14},
   EDITION = {Second},
 PUBLISHER = {American Mathematical Society, Providence, RI},
      YEAR = {2001},
     PAGES = {xxii+346},
      ISBN = {0-8218-2783-9},
   MRCLASS = {00A05 (26-01 28-01 31-01 35J10 42-01)},
  MRNUMBER = {1817225},
       DOI = {10.1090/gsm/014},
       URL = {https://doi.org/10.1090/gsm/014},
}

@article{Al92, 
title={An Application of Homogenization Theory to Harmonic Analysis: Harnack Inequalities And Riesz Transforms on Lie Groups of Polynomial Growth}, 
volume={44}, 
DOI={10.4153/CJM-1992-042-x}, 
number={4}, 
journal={Canadian Journal of Mathematics}, 
author={Alexopoulos, G.}, 
year={1992}, 
pages={691–727}
}

@article{Bar10,
  title={Riesz Transforms on Groups of Heisenberg Type.},
  author={Barbas, Helena},
  journal={Journal of Geometric Analysis},
  volume={20},
  number={1},
  pages={1--38},
  year={2010}
}

@article{BO15,
  title={Sharp martingale inequalities and applications to Riesz transforms on manifolds, Lie groups and Gauss space},
  author={Ba{\~n}uelos, Rodrigo and Os{\k{e}}kowski, Adam},
  journal={Journal of Functional Analysis},
  volume={269},
  number={6},
  pages={1652--1713},
  year={2015},
  publisher={Elsevier}
}

@article{CMZ96,
  title={About Riesz transforms on the Heisenberg groups},
  author={Coulhon, Thierry and M{\"u}ller, Detlef and Zienkiewicz, Jacek},
  journal={Mathematische Annalen},
  volume={305},
  number={1},
  pages={369--379},
  year={1996},
  publisher={Springer}
}

@article{Fol75,
  title={Subelliptic estimates and function spaces on nilpotent Lie groups},
  author={Folland, Gerald B},
  journal={Arkiv f{\"o}r matematik},
  volume={13},
  number={1},
  pages={161--207},
  year={1975},
  publisher={Kluwer Academic Publishers Dordrecht}
}

@book{FS82,
    AUTHOR = {Folland, G. B. and Stein, Elias M.},
     TITLE = {Hardy spaces on homogeneous groups},
    SERIES = {Mathematical Notes},
    VOLUME = {28},
 PUBLISHER = {Princeton University Press, Princeton, NJ; University of Tokyo
              Press, Tokyo},
      YEAR = {1982},
     PAGES = {xii+285},
      ISBN = {0-691-08310-X},
   MRCLASS = {43A85 (22E45 42B30)},
  MRNUMBER = {657581},
MRREVIEWER = {Daryl\ Geller},
}

@incollection{Gar16,
  title={Hypoelliptic operators and some aspects of analysis and geometry of sub-Riemannian spaces},
  author={Garofalo, Nicola},
  booktitle={Geometry, analysis and dynamics on sub-Riemannian manifolds},
  pages={123--257},
  year={2016},
  publisher={European Mathematical Society-EMS-Publishing House GmbH}
}

@article{GR18,
  title={Partial balayage on Riemannian manifolds},
  author={Gustafsson, Bj{\"o}rn and Roos, Joakim},
  journal={Journal de Math{\'e}matiques Pures et Appliqu{\'e}es},
  volume={118},
  pages={82--127},
  year={2018},
  publisher={Elsevier}
}

@article{Jan04,
  title={Weak-type estimates for singular integrals and the Riesz transform},
  author={Janakiraman, Prabhu},
  journal={Indiana University Mathematics Journal},
  pages={533--555},
  year={2004},
  publisher={JSTOR}
}

@book{KS00,
  title={An introduction to variational inequalities and their applications},
  author={Kinderlehrer, David and Stampacchia, Guido},
  year={2000},
  publisher={SIAM}
}

@article{Kli21,
  title={Quasi-regular Dirichlet forms and the obstacle problem for elliptic equations with measure data},
  author={Tomasz Klimsiak},
  journal={Studia Mathematica},
  volume={258},
  pages={121-156},
  year={2021}
}

@article{LP04,
  title={Riesz transforms on generalized Heisenberg groups and Riesz transforms associated to the CCR heat flow},
  author={Lust-Piquard, Fran{\c{c}}oise},
  journal={Publicacions Matem{\`a}tiques},
  pages={309--333},
  year={2004},
  publisher={JSTOR}
}

@misc{OSS26,
      title={A dimension-free weak-type $(1,1)$ bound for the vector Riesz transform on $\mathbb{R}^n$}, 
      author={Yuyuan Ouyang and Daniel Spector and Cody B. Stockdale},
      year={2026},
      eprint={2608.18068},
      archivePrefix={arXiv},
      primaryClass={math.CA},
      url={https://arxiv.org/abs/2608.18068}, 
}

@article{SS21,
  title={On the dimensional weak-type (1, 1) bound for Riesz transforms},
  author={Spector, Daniel and Stockdale, Cody B},
  journal={Communications in Contemporary Mathematics},
  volume={23},
  number={07},
  pages={2050072},
  year={2021},
  publisher={World Scientific}
}

@article{Ste83,
  title={Some results in harmonic analysis in {${\bf R}^{n}$}, for
              {$n\rightarrow \infty $}},
  author={Stein, Elias M},
  journal={Bulletin of the American Mathematical Society},
  volume={9},
  number={1},
  pages={71--73},
  year={1983}
}

@inproceedings{St86,
  title={Problems in harmonic analysis related to curvature and oscillatory integrals},
  author={Stein, Elias M},
  booktitle={Proceedings of the International Congress of mathematicians},
  volume={1},
  number={2},
  pages={196--221},
  year={1986},
  organization={Berkeley Calif.}
}

@article{SV13,
  title={Lewy--Stampacchia type estimates for variational inequalities driven by (non) local operators},
  author={Servadei, Raffaella and Valdinoci, Enrico},
  journal={Revista Matem{\'a}tica Iberoamericana},
  volume={29},
  number={3},
  pages={1091--1126},
  year={2013}
}

@book{VSC92,
  title={Analysis and geometry on groups},
  author={Varopoulos, Nicholas Th and Saloff-Coste, Laurent and Coulhon, Thierry},
  year={1992},
  publisher={Cambridge University Press}
}

@article {PV13,
    AUTHOR = {Pinamonti, Andrea and Valdinoci, Enrico},
     TITLE = {A {L}ewy-{S}tampacchia estimate for variational inequalities
              in the {H}eisenberg group},
   JOURNAL = {Rend. Istit. Mat. Univ. Trieste},
  FJOURNAL = {Rendiconti dell'Istituto di Matematica dell'Universit\`a di
              Trieste. An International Journal of Mathematics},
    VOLUME = {45},
      YEAR = {2013},
     PAGES = {23--45},
      ISSN = {0049-4704,2464-8728},
   MRCLASS = {35R03 (35H20 49J40)},
  MRNUMBER = {3168296},
}

@article {GRS22,
    AUTHOR = {Goel, Divya and R\u{a}dulescu, Vicen\c{t}iu D. and Sreenadh,
              Konijeti},
     TITLE = {Variational framework and {L}ewy-{S}tampacchia type estimates
              for nonlocal operators on {H}eisenberg group},
   JOURNAL = {Ann. Fenn. Math.},
  FJOURNAL = {Annales Fennici Mathematici},
    VOLUME = {47},
      YEAR = {2022},
    NUMBER = {2},
     PAGES = {707--721},
      ISSN = {2737-0690,2737-114X},
   MRCLASS = {35R03 (49J40)},
  MRNUMBER = {4412512},
MRREVIEWER = {Qinglong\ Zhang},
       DOI = {10.54330/afm.116794},
       URL = {https://doi.org/10.54330/afm.116794},
}

\medskip
\mbox{}\\
Sheng-Chen Mao\\
School of Mathematics and Statistics\\
Lanzhou University\\
No.~222 Tianshui South Road\\
Lanzhou 730000, P.R. China\\
Email: \href{mailto:maoshengchen@lzu.edu.cn}{maoshengchen@lzu.edu.cn}
\quad or \quad
\href{mailto:maosci@163.com}{maosci@163.com}

\bigskip\noindent
Yaojun Wang \\
Mathematisches Institut \\
Ludwig-Maximilians-Universit\"at M\"unchen \\
Theresienstr. 39 \\
80333 M\"unchen, Germany \\
Email: \href{yaojun.wang@math.lmu.de}{yaojun.wang@math.lmu.de}

\medskip
\mbox{}\\
Ye Zhang (\textit{Corresponding author})\\
SISSA\\
via Bonomea 265\\
34136 Trieste, Italy\\
Email: \href{mailto:yezhang@sissa.it}{yezhang@sissa.it}
\quad or \quad
\href{mailto:zhangye0217@gmail.com}{zhangye0217@gmail.com}

\end{document}